\documentclass[11pt]{article}
\usepackage[english]{babel}
\usepackage[utf8]{inputenc}
\usepackage{amsmath, amsfonts, amsthm, amssymb, amscd,enumerate, mathrsfs}
\usepackage{graphicx, color}
\usepackage[mathcal]{euscript}
\newtheorem{theorem}{Theorem}[section]
\newtheorem{prop}{Proposition}[section]
\newtheorem{lemma}{Lemma}[section]
\newtheorem{Cor}{Corollary}[section]
\newtheorem{nb}{Remark}[section]

\newtheorem{Def}{Definition}[section]
\numberwithin{equation}{section}
\usepackage{tikz-cd}
\usepackage{tikz}
\usepackage{float}
\usepackage[labelformat=empty]{caption}
\begin{document}
\title{On a Runge Theorem over $\mathbb{R}_3$}
\date{}
\author{Cinzia Bisi, \, Antonino De Martino, \, J\"org Winkelmann }
\maketitle
\begin{abstract}
\noindent
In this paper we investigate a topological characterization of the Runge theorem in the Clifford algebra $ \mathbb{R}_3$ via the description of the homology groups of axially
symmetric open subsets of the quadratic cone in $\mathbb{R}_3$. 
\end{abstract}
\section{Introduction}
The theory of holomorphic approximation is an important branch of mathematics. It has applications in many fields such as holomorphic dynamics, the theory of minimal surfaces in euclidean spaces, complex analysis and hypercomplex analysis \cite{GSA, GSA1}. The classical theory of holomorphic approximation began in 1885 with the work of Runge and Weierstrass. Later other mathematicians, such as Oka, Weil, Mergelyan, Vituskin, gave important results to this theory \cite{FFF}.
At the end of the 20th century the Runge theorem was studied in the hypercomplex setting. One of the first works in this direction was written by Delanghe and Brackx \cite{DB}. In this paper the authors have proved a Runge type theorem for functions which takes values in a Clifford algebra and are in the kernel of a generalized Cauchy-Riemann operator. Other results in this modern setting were obtained in the paper \cite{CSS}, in which the authors have proved the Runge approximation theorem for slice monogenic functions and slice regular functions. Recently in the paper \cite{BW1} the authors proved a topological characterization of the Runge theorem in the quaternionic setting.
\\ Inspired by this last work, in this paper we prove a Runge theorem in the Clifford algebra $\mathbb{R}_3$. In order to do this we describe the homology of axially symmetric open subsets of the quadratic cone of $\mathbb{R}_3$. Basically, we prove the following theorem
\begin{theorem}
\label{intro}
Let $D \subset D_1$ be symmetric open subsets of $\mathbb{C}$ (w.r.t the real axis) and let $ \Omega_D^4 \subset \Omega_{D_1}^4$ be the corresponding axially symmetric open subsets of the quadratic cone $ \mathcal{Q}_{\mathbb{R}_3}$. Then the following conditions are equivalent
\begin{itemize}
\item[1)] $D \subset D_1$ is a Runge pair. This means that every holomorphic function on D can be approximated by holomorphic functions on $D_1$ (uniformly on compact sets).
\item [2)] $ \Omega_D^4$ is Runge in $\Omega_{D_1}^4$, in the sense that every slice regular function on $ \Omega_D^4$ can be approximated (uniformly on compact sets) by slice regular functions on $ \Omega_{D_1}^4$.
\item [3)] $i_{*}:H_1(D) \to H_1(D_1)$ is injective, where $i_{*}$ is the homology group homomorphism induced by the inclusion map $i:D \to D_1$.
\item [4)] $i_{*}: H_{k}(\Omega_D^4) \to H_{k}(\Omega_{D_1}^4)$ is injective for $k \in \{1,3,5\}$, where $i_{*}$ is the homomorphism induced by the inclusion map $i: \Omega_D^4 \to \Omega_{D_1}^4$.
\item [5)] Every bounded connected component of $ \mathbb{C} \setminus D$ intersects $ \mathbb{C} \setminus D_1$.
\item[6)] Every bounded connected component of  $ \mathcal{Q}_{\mathbb{R}_3}\setminus \Omega_D^4$ intersects $ \mathcal{Q}_{\mathbb{R}_3} \setminus \Omega_{D_1}^4$.
\end{itemize}
\end{theorem}
One of the main differences with the quaternionic case is that in the point 4) more homological groups are affected.

The plan of the paper is the following: in Section 2 we recall some basic notions about quaternions, the Clifford algebra $ \mathbb{R}_3$ and its quadratic cone. Moreover, in this section we recall the following crucial fact
$$ \mathbb{R}_3 \simeq \mathbb{R}_2 \oplus \mathbb{R}_2,$$
where $ \mathbb{R}_2$ is the algebra of quaternions. It is important to remark that this splitting holds both at the level of real vector space and at the level of algebra.

In Section 3 we state Theorem \ref{intro}. The proof is based on proving the following equivalences

\begin{enumerate}
\item[I)] \qquad \qquad \qquad \qquad \qquad \qquad \qquad  $ 1) \Longleftrightarrow 3) \Longleftrightarrow 5),$
\item[II)] \qquad \qquad \qquad \qquad \qquad \qquad \qquad \qquad $ 1) \Longleftrightarrow 2),$
\item[III)] \qquad \qquad \qquad \qquad \qquad \qquad \qquad \qquad $ 5)\Longleftrightarrow 6),$
\item[IV)] \qquad \qquad \qquad \qquad \qquad \qquad \qquad \qquad $ 3) \Longleftrightarrow 4),$
\end{enumerate}
and the implication
$$ 6) \Longrightarrow 2),$$
which may be of interest in some other contexts. The first ones are related to the complex case. For proving the second one we show the following inequalities
$$ \frac{1}{\sqrt{2}} \| F(\alpha+i \beta) \| \leq \max \{| f(\alpha+\beta J)|, |f(\alpha-\beta J)| \} \leq \sqrt{2} \| F( \alpha+i \beta) \| \qquad  \forall \,  \alpha, \beta \in \mathbb{R}, J \in \mathbb{S}_{\mathbb{R}_{3}},$$
where $f: \Omega_D^4 \subset \mathcal{Q}_{\mathbb{R}_3} \to \mathbb{R}_3$ is a left slice function induced by a stem function $F$.
\\ The proof of the third equivalence is trivial but we show it for the sake of completeness. Due to a slice representation of the quadratic cone $ \mathcal{Q}_{\mathbb{R}_3}$ (see Proposition \ref{one}), it is possible to endow it with the topology induced by a product topology. This is helpful to prove the implication $ 6) \Longrightarrow 2)$. In order to show the last equivalence, firstly, we study the homologies $H_5(\Omega_{D}^4)$, $H_4(\Omega_{D}^4)$, $H_3(\Omega_{D}^4)$, $H_2(\Omega_{D}^4)$, $H_1(\Omega_{D}^4)$, where $\Omega_D^4$ is an axially symmetric open subset of the quadratic cone $ \mathcal{Q}_{\mathbb{R}_3}$, and after we develop a series of technical results which help us to prove the last equivalence.

\section{Preliminaries and notations}

In this section we will overview and collect the main notions and results needed for our aims.
%First of all, let us denote by $ \mathbb{R}_2$ the real algebra of the quaternions.
  First, let us recall that the skew field of quaternions
  may be identified with the Clifford algebra $ \mathbb{R}_2$.
An element $q \in \mathbb{R}_2$ is usually written as $q=x_0+ix_1+jx_2 +kx_3$, where $i^2=j^2=k^2=-1$ and $ijk=-1$. Given a quaternion $q$ we introduce a conjugation in $ \mathbb{R}_2$ (the usual one), as $q^c=x_0-ix_1-jx_2 -kx_3$; with this conjugation we define the real part of $q$ as $Re(q):= (q+q^c)/2$ and the imaginary part of $q$ as $Im(q)=(q-q^c)/2$. With the defined conjugation we can write the euclidian square norm of a quaternion $q$ as $|q|^2=qq^c$. The subalgebra of real numbers will be identified, of course, with the set $ \mathbb{R}= \{ q \in \mathbb{R}_2 \, \,| \, \, Im(q)=0 \}$.

Now, if $q$ is such that $Re(q)=0$, then the imaginary part of $q$ is such that $ (Im(q)/|Im(q)|)^2=-1$. More precisely, any imaginary quaternion $I=ix_1+jx_2 +kx_3$, such that $x^2_1+x^2_2+x^2_3=1$ is an imaginary unit. The set of imaginary units is then a real 2-sphere and it will be conveniently denoted as follows
$$ \mathbb{S}_{\mathbb{R}_2}:= \{ q \in \mathbb{R}_2\,\,| \,\, q^2=-1 \}= \{q  \in \mathbb{R}_2\,\,| \,\, Re(q)=0, \, |q|=1 \}.$$
With the previous notation, any $q \in \mathbb{R}_2$ can be written as $q= \alpha+I \beta$, where $ \alpha, \beta \in \mathbb{R}$ and $ I \in \mathbb{S}_{\mathbb{R}_2}$.
\noindent 
Given any $I \in \mathbb{S}_{\mathbb{R}_2}$ we will denote the real subspace of $ \mathbb{R}_2$ generated by 1 and $I$ as
$$ \mathbb{C}_I:= \{ q \in \mathbb{R}_2 \,\, | \,\, q= \alpha+I \beta, \, \, \, \alpha, \beta \in \mathbb{R} \}.$$
Sets of the previous kind will be called \emph{slices}
and they are also complex planes with respect to the complex
  structure defined by the respective parameter $I$.
All these notations reveal now clearly the \emph{slice} structure of $ \mathbb{R}_2$ as a union of complex planes $ \mathbb{C}_I$ for $I$ which varies in $ \mathbb{S}_{\mathbb{R}_2}$, i.e.
$$ \mathbb{R}_2= \bigcup_{I \in \mathbb{S}_{\mathbb{R}_2}}\mathbb{C}_{I}, \quad \bigcap_{I \in \mathbb{S}_{\mathbb{R}_2}} \mathbb{C}_I= \mathbb{R}.$$

The following notion of slice regularity was introduced by Gentili and Struppa \cite{GS,GSS}.

\begin{Def}
Let $ \Omega$ be an open subset of $ \mathbb{R}_2$ with $\Omega \cap \mathbb{R} \neq \emptyset$. A real differentiable function $f: \Omega \mapsto \mathbb{R}_2$ is slice regular if for every $I \in \mathbb{S}_{\mathbb{R}_2}$ its restriction $f_I$ to the complex plane $\mathbb{C}_I$ passing through the origin and containing $1$ and $I$ is holomorphic on $\Omega \cap \mathbb{C}_I$.
\end{Def}

For a ball in $\mathbb{R}_2$ centered at the origin we have that a slice regular function can be represented by a convergent power series
$$ f(q)= \sum_{k=0}^{+ \infty} q^k a_k, \qquad \{a_k\}_{k \in \mathbb{N}} \subset \mathbb{R}_2.$$

The theory of slice regular functions has given already many fruitful results, both on the analytic and the geometric side, see for example \cite{ANB, AB, BAdM, BG1, BG2, BG3, BS1, BS2, BS3, BW, BW2}.

Moreover, slice hyperholomorphic functions have several applications in operator theory and in Mathematical Physics \cite{CG,CGK,GMP}.
The spectral theory of the S-spectrum is a natural tool for the formulation of quaternionic quantum mechanics  and for the study of new classes of fractional diffusion problems, see \cite{CG, CGK, CSS1}, and the references therein. Slice hyperholomorphic functions are also important in operator theory and Schur analysis which have also been deeply investigated in the recent years, see \cite{ACS,ACS1} and the references therein.

Now we will see some basic notions about the real Clifford algebra $ \mathbb{R}_3$ and its quadratic cone $\mathcal{Q}_{\mathbb{R}_3}$, introduced in the papers \cite{GP, GP1}.
\\ We define $\mathbb{R}_3$ as the real associative non-commutative algebra defined as follows. Let $ \{ e_1, e_2, e_3 \}$ be the canonical orthonormal basis for $ \mathbb{R}^3$.
  Then $\mathbb{R}_3$ is the real associative algebra with $1$ generated
  by the $e_i$ 
with defining relations $ e_{i}e_{j}+e_{j}e_{i}=-2 \delta_{ij}$.
  This is the real Clifford algebra for the vector space $\mathbb{R}^3$
  with the standard euclidean  quadratic form.
In the sequel we will write $e_0:=1$, $e_{i} e_{j}= e_{ij}$, for $i,j=1,2,3$, $i \neq j$, and $e_{1}e_{2}e_{3}=e_{123}.$ Thus an arbitrary element $x \in \mathbb{R}_3$ can be written as
\begin{equation}\label{class}
 x=x_0e_0+x_1e_1+x_2e_2+x_3e_3+x_{12}e_{12}+x_{13}e_{13}+x_{23}e_{23}+x_{123}e_{123}
\end{equation}
where the coefficients $x_i$, $x_{ij}$, $x_{ijk}$ are real numbers. Thus, we see that $\mathbb{R}_3$ is an eight dimensional real space, endowed with a natural multiplicative structure.
The conjugate of $x$ will be denoted by $\bar{x}$
and can be defined
as the unique antivolution%
\footnote{An 
  is a linear self-map of order $2$ such that
  $\overline{xy}=(\bar y)\cdot (\bar x)\ \forall x,y \  \in A $, with $A$ any real quadratic alternative algebra with a unity}
 of ${\mathbb R}_3$
with $e_i\mapsto \bar{e_i}=-e_i$.
Conjugation may likewise be defined extending by linearity the anti-involution 
$$\overline{e_0}=e_0, \quad \overline{e_i}=-e_i, \quad \overline{e_{ij}}=-e_{ij}, \quad  \overline{e_{123}}= e_{123},$$ 
for $i,j \in \{1,2,3 \}, \, i \neq j. $
%\changed{The center is generated by $e_0$ and $e_{123}$.}

Moreover, it is known that in $ \mathbb{R}_3$ one can consider the two idempotents $ \omega_{+}= \frac{1}{2}(e_0+e_{123})$ and $ \omega_{-}= \frac{1}{2}(e_0-e_{123})$ (i.e. $\omega^2_{+}=\omega_{+}$ , $\omega^2_{-}=\omega_{-}$), that are mutually annihilating each other i.e. $ \omega_{+} \omega_{-}=\omega_{-} \omega_{+}=0$ (see \cite[Chapter 6]{CSS2}, \cite{DS}).
  Let $ \mathbb{R}_3^{+}$ denote the even subalgebra of $ \mathbb{R}_3$
  i.e.
$$ \mathbb{R}_3^{+}= \{ x_{0}e_{0}+x_{23}e_{23}+ x_{12} e_{12}+x_{13}e_{13}: \, x_{0},x_{23},x_{12}, x_{12} \in \mathbb{R} \}.$$
Note that $\mathbb{R}_3^{+}\simeq\mathbb{R}_2$ as $\mathbb{R}$-algebras.

Every $x \in \mathbb{R}_3$ admits a unique representation
\begin{equation}\label{split}
x=\omega_+q+\omega_-p
\end{equation}
with $q,p \in \mathbb{R}_{3}^{+} \simeq \mathbb{R}_2$.
So we have the isomorphism of $\mathbb{R}$-algebras \cite{DSS}
\begin{equation}
\label{ssplit}
\mathbb{R}_3 =\omega_+\mathbb{R}_3^{+}
  \oplus\omega_-\mathbb{R}_3^{+}\simeq \mathbb{R}_2 \oplus \mathbb{R}_2
\end{equation}
where the ring structure is given by $(q,p)+(q',p')=(q+q', p+p')$ and $(q,p)(q', p')=(qq',pp')$.
The equality \eqref{split} is very useful since helps us to work in the Clifford algebra $ \mathbb{R}_3$ using the quaternionic results.

Conjugation on $ \mathbb{R}_3$ is compatible with conjugation
on $\mathbb{R}_2$ via this splitting:
\[
\bar x=\omega_{+} q^c+ \omega_{-} p^c, \quad  \overline{\omega_{+}}=\omega_{+}, \quad  \overline{\omega_{-}}=\omega_{-}
\]
for $x=\omega_+q+\omega_-p$
where $q^c$ and $p^c$ are conjugate of $q$
and $p$ as elements of $\mathbb{R}_2$.

As usual, we have the notions of {\em norm} and {\em trace} associated
to the conjugation, i.e., the norm $n(x)$ is defined as $x\bar x$ and
the trace $t(x)$ is defined as $t(x)=x+\bar x$.

With respect to the splitting
$x=\omega_+q+\omega_-p$ we obtain:
\begin{equation}
\label{split-formula}
t(x)=\omega_+t(q)+\omega_-t(p),\quad
n(x)=\omega_+n(q)+\omega_-n(p)
\end{equation}

The norm is multiplicative, i.e., $n(xy)=n(x)n(y), \, \forall x,y\in
{\mathbb R}_3$.

For more details about this splitting the interested reader can see \cite{DSS, R1,R,SS}.
\begin{nb}
  In general, it is known that every Clifford algebra ${\mathbb R}_n$
  is a either a matrix algebra of rank $r\ge 1$ over 
  $ \mathbb{R}$, $ \mathbb{C}$ or $ \mathbb{H}$
  or a direct sum of two copies of such a matrix algebra
  \cite{LM, P}.

  An explicit proof of the splitting
  \[
  \mathbb{R}_3\simeq Mat(1\times 1,\mathbb{H})\oplus Mat(1\times 1,\mathbb{H})
  \simeq \mathbb{R}_2\oplus \mathbb{R}_2
  \]
  may be found in the papers \cite{R1,R}. However, the reader should be
  aware
  these papers also contain an incorrect claim that $ \mathbb{R}_n$ is isomorphic to a sum of $2^{n-1}$ copies of the algebra $ \mathbb{R}_2$.
  \end{nb}
Now, we introduce some basic facts about the quadratic cone \cite{GP, GP1}.
\begin{Def}[\cite{GP}]
We call quadratic cone of $ \mathbb{R}_3$ the set
$$ \mathcal{Q}_{\mathbb{R}_3}:= \mathbb{R} \cup \{ x \in \mathbb{R}_3 \setminus \mathbb{R} \, \, | \, t(x) \in \mathbb{R}, \, n(x) \in \mathbb{R}, \, 4n(x)> t(x)^2 \}.$$
\end{Def}
\begin{nb}
This definition has been introduced for general finite-dimensional real alternative algebras.
 The inequality $4n(x)> t(x)^2 $ is relevant for the general case,
  but not in our case.
  In fact, one can check that in $\mathbb{R}_3$  the inequality
  $4n(x)> t(x)^2$ is automatically fulfilled as soon as
  $n(x),t(x)\in\mathbb{R}$.
  Indeed, due to \eqref{split-formula}, the assumption
  $n(x),t(x)\in\mathbb{R}$ implies $n(x)=n(p)=n(q)$,
  $t(x)=t(p)=t(q)$ for $x=\omega_+q+\omega_-p$ with
  $p\in\mathbb{R}_3^+\simeq\mathbb{H}$
  Thus it suffices to check the inequality for quaternions,
  Every quaternion $q$ may be written as $q=s+v$ where $s\in\mathbb{R}$ and $v$
  is imaginary. Then $n(q)=s^2+||v||^2$ and $t(q)=2s$, implying
  $4n(q)>t(q)^2$ if $v\ne 0$.
  
  As a consequence,
$$ \mathcal{Q}_{\mathbb{R}_3}:= \mathbb{R} \cup \{ x \in \mathbb{R}_3 \setminus \mathbb{R} \, \, | \, t(x), n(x) \in \mathbb{R}\}.$$
\end{nb}

In terms of the splitting of ${\mathbb R}_3$ we have:
\[
\mathcal{Q}_{\mathbb{R}_3}
=\left\{ \omega_+q+\omega_-p: p,q\in {\mathbb R}_2, t(p)=t(q), n(p)=n(q)
\right\}
\]

We also define $ \mathbb{S}_{\mathbb{R}_3}:= \{x \in \mathcal{Q}_{\mathbb{R}_3}| \, x^2=-1 \}. $ The elements of $ \mathbb{S}_{\mathbb{R}_3}$ will be called \emph{square roots} of $-1$ in the algebra $ \mathbb{R}_3$.

Next we show
$ \mathbb{S}_{\mathbb{R}_3}\simeq\mathbb{S}_{\mathbb{R}_2} \times \mathbb{S}_{\mathbb{R}_2}$:
\begin{prop}
\label{split2}
\[
\mathbb{S}_{\mathbb{R}_3}
=\left\{ \omega_+q+\omega_-p: p,q\in {\mathbb R}_2, q^2=p^2=-1
\right\}
\]
\end{prop}
\begin{proof}
Let $ x= \omega_{+}q+\omega_{-}p$ ($q,p \in \mathbb{R}_2$).
By elevating to the square and using the facts that $ \omega^2_+=\omega_+$, $\omega^2_-= \omega_-$ and $ \omega_+\omega_-=\omega_-\omega_+=0$ we get
$$ x^2=\omega^2_+ q^2+ \omega_+\omega_-qp+\omega_-\omega_+pq+\omega^2_-p^2=\omega_{+} q^2+\omega_{-} p^2.$$
Since $ \omega_{\pm}= \frac{1}{2}(e_0 \pm e_{123})$ we get
$$ x^2=-1 \iff -(\omega_++\omega_-)=-1=x^2=\omega_{+} q^2+\omega_{-} p^2
\iff p^2=-1=q^2 $$
Thus
\[
\{ x \in\mathbb{R}_3:x^2=-1\}
=
\{\omega_+q+\omega_-p: q^2=p^2=-1\}
\]
We observe that $q^2 =p^2=-1$ implies $n(x)=1\in\mathbb{R}$ and $t(x)=0\in\mathbb{R}$.
Therefore
\[
\{ x \in\mathbb{R}_3:x^2=-1\}=
\{ x \in\mathbb{R}_3:x^2=-1,n(x),t(x)\in\mathbb{R}\}
=\{ x \in \mathcal{Q}_{\mathbb{R}_3}:x^2=-1\}
= \mathbb{S}_{\mathbb{R}_3}
\]

This means that $ \mathbb{S}_{\mathbb{R}_3}\cong \mathbb{S}_{\mathbb{R}_2} \oplus \mathbb{S}_{\mathbb{R}_2}\cong\mathbb{S}_{\mathbb{R}_2} \times \mathbb{S}_{\mathbb{R}_2}$.
\end{proof}

The following proposition, proved in \cite[Prop. 3]{GP}, will be important for our results. 
\begin{prop}
\label{one}
The following statements hold
\begin{enumerate}
\item $\mathcal{Q}_{\mathbb{R}_3}= \bigcup_{J \in \mathbb{S}_{\mathbb{R}_3}} \mathbb{C}_J,$
\item If $I,J \in \mathbb{S}_{\mathbb{R}_3}$, $I \neq \pm J$, then $ \mathbb{C}_I \cap \mathbb{C}_J= \mathbb{R}$.
\end{enumerate}
\end{prop}
In \cite{GP} the authors studied the quadratic cone for a general finite-dimensional real alternative algebra with unity. They remark that if the algebra is isomorphic to one of the division algebras $\mathbb{R}_2$, $\mathbb{O} $  we have that $ \mathcal{Q}_{\mathbb{R}_2}= \mathbb{R}_2$ and $ \mathcal{Q}_\mathbb{O}= \mathbb{O}$. Furthermore, in these cases, $\mathbb{S}_{\mathbb{R}_2}= \{ q \in \mathbb{R}_2 : \, q^2 =-1 \}$, i.e. is a 2-sphere, and $ \mathbb{S}_{\mathbb{O}}$ is a 6-sphere.

In the case of the Clifford algebra $ \mathbb{R}_n$, for $ n\geq 3$, the quadratic cone $ \mathcal{Q}_{\mathbb{R}_n}$ is a real algebraic subset of $ \mathbb{R}_n$. Now, we recall that an element $(x_0, x_1,...,x_n) \in \mathbb{R}^{n+1}$ can be identified with the element
$$ x=x_0 e_0+ \sum_{j=1}^{n} x_je_j,$$
called, in short paravector. We remark that the subspaces of paravectors $ \mathbb{R}^{n+1}$ is contained in $\mathcal{Q}_{\mathbb{R}_n}$. Moreover, the $(n-1)$- dimensional sphere $ \mathbb{S}= \{x=x_{1}e_1+...+x_ne_n \in \mathbb{R}^{n}| \, x_1^2+...+x_n^2=1 \}$ of unit imaginary vectors is properly contained in $ \mathbb{S}_{\mathbb{R}_n}$.

In particular, for the case $ \mathbb{R}_3$ it is possible to show that the quadratic cone is the 6-dimensional real algebraic set
$$ \mathcal{Q}_{\mathbb{R}_3}= \{x \in \mathbb{R}_3 \,|  \,x_{123}=0, \, x_2x_{13}-x_1x_{23}-x_3 x_{12}=0\}.$$
Moreover $ \mathbb{S}_{\mathbb{R}_3}$ is the intersection of a 5-sphere with the hypersurface $x_1x_{23}-x_2x_{13}+x_{3}x_{12}=0$.

For our future computations it will be essential the following definition.
\begin{Def}
\label{ax}
Let us consider an open subset $D$ of $ \mathbb{C}$, we define $ \Omega_D^4$ as a subset of $\mathbb{R}_3$ such that
$$ \Omega_D^4:= \{x= \alpha+\beta J \in \mathbb{C}_J|\,\alpha, \beta \in \mathbb{R}, \, \alpha+i \beta \in D, \, J \in \mathbb{S}_{\mathbb{R}_3} \}.$$
This kind of set will be called axially symmetric domain. Furthermore, we set
$$ \Omega_D^2:=\{x= \alpha+\beta K \in \mathbb{C}_K|\,\alpha, \beta \in \mathbb{R}, \, \alpha+i \beta \in D, \, K \in \mathbb{S}^2 \},$$
where in this case $ \mathbb{S}^2$ is a generic 2-sphere, which is contained in $ \mathbb{S}_{\mathbb{R}_3}$ (see Proposition \ref{split2}).
\end{Def}
If $ \mathbb{S}^2= \mathbb{S}_{\mathbb{R}_2}$ we obtain the axially symmetric domain of the quaternions, (see \cite{GSS}). For this we will use the following notation $ \Omega_{D}^2(\mathbb{R}_2)$.
From \cite{BD2} we have a relation between the axyally symmetric domains previously defined.
\begin{prop}
\label{np}
Let us consider an open subset $D$ of $ \mathbb{C}$, then we have
$$ \Omega_{D}^4 \cong \Omega_{D}^2(\mathbb{R}_2) \oplus \Omega_{D}^2(\mathbb{R}_2).$$
\end{prop}
By Proposition \ref{one} follows that $\Omega_D^4$ is a relatively open subset of the quadratic cone $ \mathcal{Q}_{\mathbb{R}_3}$.
\\We define $ \mathbb{R}_3 \otimes \mathbb{C}$ as the complexification of $ \mathbb{R}_3$. We will use the representation
$$ \mathbb{R}_3 \otimes \mathbb{C}= \{ w=x+\iota y|\, x,y \in \mathbb{R}_3 \} \qquad (\iota^2=-1).$$
\begin{Def}
\label{slice}
A function $F: D \subseteq \mathbb{C} \to \mathbb{R}_3 \otimes \mathbb{C}$ defined by $F(z)= F_1(z)+i F_2(z)$ where $z= \alpha+i \beta \in D$, $F_1,F_2: D \to \mathbb{R}_3$, and where $F_1(\bar{z})=F_1(z)$ and $F_2(\bar{z})=-F_2(z)$ whenever $z, \bar{z} \in D$, is called stem function. Given a stem function $F$, the function $f= \mathcal{I}(F)$ defined by
$$ f(x)=f(\alpha+ \beta J):=F_1(z)+JF_2(z)$$
for any $ x \in \Omega_D^4$ is called the slice function induced by $F$.
\end{Def}
We will denote the set of (left) slice functions on $ \Omega_D^4$ by
$$ \mathcal{S}(\Omega_D^4):= \{ f: \Omega_D^4 \to \mathbb{R}_3| \,  f=\mathcal{I}(F), \, F:D \to \mathbb{R}_3 \otimes \mathbb{C} \, \, \, \hbox{stem function} \}.$$

\begin{prop}
\label{Rap1}
Let $f \in \mathcal{S}(\Omega_D^4)$ and $J \in \mathbb{S}_{\mathbb{R}_3}$. Then the following equality holds for all $x= \alpha+ \beta I \in \Omega_D^4 \cap \mathbb{C}_I$
\begin{equation}\label{Rap}
f(x)= \frac{1}{2} [f( \alpha+ \beta J)+f( \alpha-\beta J)]+ \frac{I}{2}\bigl[J[ f(\alpha- \beta J)-f(\alpha+\beta J)] \bigl]= F_1(\alpha, \beta)+I F_2(\alpha, \beta).
\end{equation}
Moreover, $F_1$ and $F_2$ depend only on $\alpha$, $\beta$ but they do not depend on $J \in \mathbb{S}_{\mathbb{R}_3}$.
\end{prop}
\begin{proof}
In \cite[Prop. 6]{GP} the authors prove a representation formula for slice functions in a finite-dimensional real alternative algebra with unity. The second part of the proof follows directly by \eqref{Rap}. In fact we have
\begin{eqnarray*}
&& \frac{1}{2}[f(\alpha+\beta J)+f(\alpha- \beta J)]= \frac{1}{2} \biggl[ \frac{1}{2} [f( \alpha+\beta I)+f( \alpha- \beta I)]+ \frac{JI}{2}[ f(\alpha- \beta I)-f(\alpha+\beta I)]+\\
&& + \frac{1}{2} [f( \alpha+\beta I)+f( \alpha- \beta I)]- \frac{JI}{2}[ f(\alpha- \beta I)-f(\alpha+\beta I)] \biggl]=\frac{1}{2}[f(\alpha+\beta I)+f(\alpha- \beta I)],
\end{eqnarray*}
and so $F_1$, and similarly $F_2$, depend on $ \alpha$, $ \beta$ only.
\end{proof}
\begin{Def}
A (left) slice function $f \in \mathcal{S}(\Omega_D^4)$ is (left) slice regular if its stem function $F$ is holomorphic.
\end{Def}
We will denote the set of slice regular functions on $\Omega_D^4$ by
$$ \mathcal{SR}(\Omega_D^4)= \{ f \in \mathcal{S}(\Omega_D^4)| \, f= \mathcal{I}(F), \, F:D \to \mathbb{R}_{3} \otimes \mathbb{C} \, \, \hbox{holomorphic} \}.$$

In \cite[Prop.1]{GS1} the authors showed that $\mathbb{S}_{\mathbb{R}_3}$ is the intersection of a 5-sphere with the hypersurface $x_1x_{23}-x_2x_{13}+x_{3}x_{12}=0$.

\begin{nb}
In the Clifford algebra $ \mathbb{R}_3$ it is possible to define the sphere $ \mathbb{S}^6$ of pure imaginary  units (i.e. the set in which the square of the imaginary units is equal to minus one). Unlike what happens over the quaternions, the sphere $ \mathbb{S}_{\mathbb{R}_{3}}$ does not coincide with $ \mathbb{S}^6$, but it is properly contained. For instance, the imaginary unit $e_{123}$ stays in $ \mathbb{S}^6$ but not in $ \mathbb{S}_{\mathbb{R}_3}$. Moreover, we observe that we have defined the axially symmetric domain using the smallest sphere (see Definition \ref{ax}).
\end{nb}

\section{Runge's Theorem}
In this section we study the Clifford analogous of the classical complex Runge theory. In particular, we focus on the extension in the Clifford algebra $ \mathbb{R}_3$ of the topological characterization of this theorem. A quaternionic generalization of this theorem is showed in \cite[Thm. 1.3]{BW1}. From now on the set $D$ is a symmetric open subset of $ \mathbb{C}$ with respect to the real line.
\begin{theorem}
\label{main}
Let $D \subset D_1$ be symmetric open subsets of $\mathbb{C}$ and let $ \Omega_D^4 \subset \Omega_{D_1}^4$ be the corresponding axially symmetric open subsets of $ \mathcal{Q}_{\mathbb{R}_3}$ (as defined in Definition \ref{ax}). Then the following conditions are equivalent
\begin{itemize}
\item[1)] $D \subset D_1$ is a Runge pair. This means that every holomorphic function on D can be approximated by holomorphic functions on $D_1$ (uniformly on compact sets).
\item [2)] $ \Omega_D^4$ is Runge in $\Omega_{D_1}^4$, in the sense that every slice regular function on $ \Omega_D^4$ can be approximated (uniformly on compact sets) by slice regular functions on $ \Omega_{D_1}^4$.
\item [3)] $i_{*}:H_1(D) \to H_1(D_1)$ is injective, where $i_{*}$ is the homology group homomorphism induced by the inclusion map $i:D \to D_1$.
\item [4)] $i_{*}: H_{k}(\Omega_D^4) \to H_{k}(\Omega_{D_1}^4)$ is injective for $k \in \{1,3,5\}$, where $i_{*}$ is the homomorphism induced by the inclusion map $i: \Omega_D^4 \to \Omega_{D_1}^4$.
\item [5)] Every bounded connected component of $ \mathbb{C} \setminus D$ intersects $ \mathbb{C} \setminus D_1$.
\item[6)] Every bounded connected component of  $ \mathcal{Q}_{\mathbb{R}_3}\setminus \Omega_D^4$ intersects $ \mathcal{Q}_{\mathbb{R}_3} \setminus \Omega_{D_1}^4$.
\end{itemize}
\end{theorem}
\begin{nb}
The main difference with the quaternionic case is that more homological groups are affected. Indeed, in \cite[Thm. 1.3]{BW1} the map $i_{*}$ is injective only for $k \in \{1,3 \}$.
\end{nb}
We prove Theorem \ref{main} by proving the following equivalences:
$$ 1) \Longleftrightarrow 3) \Longleftrightarrow 5),$$
$$ 1) \Longleftrightarrow 2),$$
$$ 5)\Longleftrightarrow 6),$$
$$ 3) \Longleftrightarrow 4),$$
and the implication $$ 6) \Longrightarrow 2),$$
which may be of interest in some other contexts. We start by proving the classical ones: $ 1) \Longleftrightarrow 3) \Longleftrightarrow 5).$
\\These equivalences are well-known in the complex case (see \cite{BS} and \cite[Paragraph 13.2.1]{RE}).
\begin{prop}
Let $D \subset D_1$ be open subsets of $ \mathbb{C}$. Then the following properties are equivalent
\begin{itemize}
\item [1)] The inclusion map induces an injective group homomorphism $H_1(D) \to H_1(D_1)$.
\item [2)] Every bounded connected component of $ \mathbb{C} \setminus D$ intersects $ \mathbb{C} \setminus D_1$.
\item[3)] For every holomorphic function $f$ on $D$, every $ \varepsilon>0$ and every compact subset $K \subset D$ there exists a holomorphic function $g$ on $D_1$ with $\sup_{p \in K}|f(p)-g(p)| < \varepsilon$.
\end{itemize}
\end{prop}
Now, we prove the implication $ 1) \Longleftrightarrow 2)$ of Theorem \ref{main}.
\\Before, we show an auxiliary result, which is similar to \cite[Lemma 2.3]{BW1}.
\begin{lemma}
\label{imp1}
Let $f: \Omega_D^4 \subset \mathcal{Q}_{\mathbb{R}_3} \to \mathbb{R}_3$ be a left slice function induced by a stem function $F$. Then
\begin{equation}\label{ineq1}
\frac{1}{\sqrt{2}} \| F(\alpha+i \beta) \| \leq \max \{| f(\alpha+\beta J)|, |f(\alpha-\beta J)| \} \leq \sqrt{2} \| F( \alpha+i \beta) \|
\end{equation}
for every $ \alpha, \beta \in \mathbb{R}$, $J \in \mathbb{S}_{\mathbb{R}_{3}}$.
\end{lemma}

\begin{prop}
Let $D \subset D_1$ be a symmetric open subset of $ \mathbb{C}$ with corresponding axially symmetric open subsets $ \Omega_D^4 \subset \Omega_{D_1}^4$ in $ \mathcal{Q}_{\mathbb{R}_3}$. Then every slice regular function on $ \Omega_D^4$ may be approximated locally uniformly by slice regular functions on $ \Omega_{D_1}^4$ if and only if $D$ is Runge in $D_1$.
\end{prop}
\begin{proof}
Firstly, we notice that for any symmetric subset $C \subset D$ the subset
$$ \Omega_C^4= \{\alpha+\beta J: \exists \, \,\alpha+\beta i\in C, \, J \in \mathbb{S}_{\mathbb{R}_3} \}$$
of $ \mathcal{Q}_{\mathbb{R}_3}$ is compact if and only if $C$ is compact. We measure the size of a function by using the sup-norm. From the euclidean scalar product on $ \mathbb{C} \simeq \mathbb{R}^2$ and on $ \mathbb{R}_3 \simeq \mathbb{R}^8$ we deduce a scalar product on $ \mathbb{R}_3 \otimes \mathbb{C} \simeq \mathbb{R}^{16}$. We call the norm induced by this scalar product by $ \|.\|$. From the inequalities \eqref{ineq1} we derive that
$$ \frac{1}{\sqrt{2}} \| F \|_C \leq \| f \|_{\Omega_C^4} \leq \sqrt{2} \| F \|_C,$$
for any compact symmetric subset $C$ of $D$. We denote $ \| F \|_C= \sup_{z \in C}\|F(z)\|$. We endow the space of slice functions on $ \Omega_D^4$ and the space of stem functions on $D$ with the topology of locally uniform convergence. Thus, by the previous inequalities we get that the two spaces are isomorphic. This implies the thesis.
\end{proof}

Now, we prove $ 5) \Longleftrightarrow 6)$ of Theorem \ref{main}.
\\This implication is trivial and it is similar to \cite[Thm. 1.3]{BW1}. For the sake of completeness we show the proof. Each bounded connected component $C$ of $D$, respectively $D_1$, corresponds to a bounded connected component $ \Omega_C^4$ of $\Omega_D^4$, respectively $\Omega_{D_1}^4$, through
$$ \Omega_C^4= \{ \alpha+\beta I; \, \alpha, \beta \in \mathbb{R}, \, \alpha+ \beta i \in C, \, I \in \mathbb{S}_{\mathbb{R}_3} \}.$$
In order to show the implication $ 6) \Longrightarrow 2)$ of Theorem \ref{main} we need a preliminary result. Let us start with the followings
\begin{Def}
Let $ D \subset \mathbb{C}$. Let $ \Omega_D^4 \subset \mathcal{Q}_{\mathbb{R}_3}$ be axially symmetric and let $f$ be a slice function. It is called intrinsic if 
$$f(\alpha-I \beta)= \overline{f( \alpha+I \beta)}, \qquad \, \forall \, \alpha+I \beta \in \Omega_D^4 \cap \mathbb{C}_I.$$
\end{Def}
\begin{Def}
Given two polynomials $ a= \mathcal{I}(A)$ and $b=\mathcal{I}(B)$ such that $A^cA$ is real and not identically zero, we call (left) rational a function of the form $a^{-1}b:= \mathcal{I}\left((A^cA)^{-1}A^c B \right).$
\end{Def}
\begin{nb}
The function $A^cA$ is said to be real if its components are real-valued.
\end{nb}

Rational functions admit the following characterization:
\begin{prop}
\label{nb1}
A slice function $f:\Omega_D^4 \to \mathbb{R}_3$ is rational if and only if for any $I \in \mathbb{S}_{\mathbb{R}_3}$ and any choice of $I_2$, $I_3$, completion of a basis of $ \mathbb{R}_3$, satisfying the defining relations, $I_rI_s+I_sI_r=-2 \delta_{rs}$. Then there exist eight rational intrinsic functions $R_A$ such that
\begin{equation}
f_{I}(u+Iv)= \sum_{|A|=0}^3 R_A(\alpha+I \beta)I_A, \quad I_A=I_{i_1}...I_{i_s} \qquad \forall \,  \alpha+I\beta \in \Omega_D^4 \cap \mathbb{C}_I,
\end{equation}
where $A=i_{1}...i_{s}$ is a subset of $ \{1,2,3\}$, with $i_1<...<i_s$ or, when $|A|=0$, $I_{\emptyset}=1$.
\end{prop}
\begin{proof}
It is based on computations similar to \cite[Prop. 3.7]{CSS} and on the refined splitting lemma, see \cite[Lemma 2.4]{GP3}, \cite{CGS}.
\end{proof}
Due to the fact that the function $A^cA$ is real and not identically zero we have the following
\begin{Cor}
The singularities of a rational function are all poles.
\end{Cor}
Now, we are ready to show the implication $ 6) \Longrightarrow 2)$ of Theorem \ref{main}, which is not necessary but could be of interest in some other contexts.
In order to do this we have to endow the quadratic cone $ \mathcal{Q}_{\mathbb{R}_3}$ with a topology induced by a product topology. For this purpose we use the fact that we can split $\mathbb{S}_{\mathbb{R}_3}$ as direct sum of 2-spheres (see Proposition \ref{split2}) and by Proposition \ref{one} we get
$$ \mathcal{Q}_{\mathbb{R}_3}=\bigcup_{J \in \mathbb{S}_{\mathbb{R}_3}} \mathbb{C}_{J}=\bigcup_{J \in \mathbb{S}_{\mathbb{R}_2} \times \mathbb{S}_{\mathbb{R}_2}} \mathbb{C}_{J}.$$
This allows us to endow the quadratic cone $ \mathcal{Q}_{\mathbb{R}_3}$ with the topology induced by the product topology of $\mathbb{C} \times \mathbb{R}^6$. It is possible to justify this choice since we are working with complex planes in which the imaginary unit varies in a product of spheres which are subset of $ \mathbb{R}_3$, respectively.
By the symbol $\overline{\mathcal{Q}_{\mathbb{R}_3}}$ we will denote $\mathcal{Q}_{\mathbb{R}_3} \cup \{ \infty \}$. 
\begin{theorem}
\label{th1}
Let $D \subset \mathbb{C}$. Let $ \Omega_D^4$ be an axially symmetric open set in $\mathcal{Q}_{\mathbb{R}_3}$, let $A$ be a set having a point in each connected component of $ \overline{\mathcal{Q}_{\mathbb{R}_3}} \setminus \Omega_D^4$ and let $f \in \mathcal{SR}(\Omega_D^4)$. Then $f$ can be approximated by a sequence of rational functions $ \{ r_n\}$ having their poles in A uniformly on every compact set in $ \Omega_D^4$. If $ \overline{\mathcal{Q}_{\mathbb{R}_3}} \setminus \Omega_D^4$ is a connected set, then we can set $A= \{ \infty \}$ and $f$ can be approximated by polynomials uniformly on every compact set in $ \Omega_D^4$.
\end{theorem}
\begin{proof}
We follow the same ideas of \cite[Thm. 3.12]{CSS}. The following estimates are to be understood in a compact set of $ \mathcal{Q}_{\mathbb{R}_3}$. It is possible to extend them for an open set through an exhaustion by compact sets, as made in \cite[Thm. 3.14]{CSS}.
\\Since the quadratic cone is endowed with the topology induced by the product topology of $\mathbb{C} \times \mathbb{R}^6$ we have to prove that for any $ \varepsilon >0$
\begin{equation}
\label{i1}
\| f(\alpha+I \beta)-r_n(\alpha+I \beta) \|_{\mathbb{C}}< \varepsilon,
\end{equation}
and
\begin{equation}
\label{i2}
\| f(\alpha+I \beta)-r_n( \alpha+I \beta) \|_{\mathbb{R}^6}< \varepsilon \qquad \forall \, \alpha+I \beta \in \Omega_D^4 \cap \mathbb{C}_I.
\end{equation}
By the refined splitting lemma \cite[Lemma 2.4]{GP3} we have that there exist eight holomorphic intrinsic functions $F_A: \Omega_D^4 \cap \mathbb{C}_I \to \mathbb{C}_I$ such that for every $z= \alpha+I \beta$
$$ f_I(z)= \sum_{|A|=0}^3 F_{A}(z)I_A.$$ 
By the Proposition \ref{nb1} and the classical Runge's theorem \cite[Thm. 13.9]{RU} we can find eight rational intrinsic functions $R_A( \alpha+I \beta)$ with poles in $A \cap \mathbb{C}_I$ such that
\begin{equation}
\label{nb2}
\| F_A(\alpha+I \beta)-R_{A}( \alpha+I \beta) \|_{\mathbb{C}} < \frac{\varepsilon}{8} \qquad  \forall \, \alpha+I \beta \in \Omega_D^4 \cap \mathbb{C}_I.
\end{equation}
Then 
$$ \| f(\alpha+I \beta)-r_n(\alpha+I \beta) \|_{\mathbb{C}}< \varepsilon.$$
In order to prove \eqref{i2} we use the Representation Formula (Proposition \ref{Rap1})
\begin{eqnarray*}
f(\alpha+I \beta)-r_n(\alpha+I \beta) &=& \frac{1}{2} \biggl[ \bigl( f(\alpha+J \beta)+f(\alpha -J \beta) \bigl)+ IJ \bigl( f(\alpha-J \beta)-f(\alpha +J \beta) \bigl)+\\
&& - \bigl( r_n(\alpha+J \beta)+r_n(\alpha -J \beta) \bigl)+  IJ \bigl( r_n(\alpha-J \beta)-r_n(\alpha +J \beta) \bigl) \biggl].
\end{eqnarray*}
Using the previous splittings we get
\begin{eqnarray*}
f(\alpha+I \beta)-r_n(\alpha+I \beta) &=& \frac{1}{2} \biggl[ \biggl( \sum_{|A|=0}^3 F_{A}( \alpha+J \beta)I_A+F_{A}( \alpha-J \beta)I_A+\\
&& +IJ\sum_{|A|=0}^3 F_{A}( \alpha-J \beta)I_A-F_{A}( \alpha+J \beta)I_A \biggl)+\\
&& - \biggl( \sum_{|A|=0}^3 R_{A}( \alpha+J \beta)I_A+R_{A}( \alpha-J \beta)I_A+\\
&& +IJ\sum_{|A|=0}^3 R_{A}( \alpha-J \beta)I_A-R_{A}( \alpha+J \beta)I_A \biggl) \biggl].
\end{eqnarray*} 
Finally by \eqref{nb2} we get
\begin{eqnarray*}
\| f(\alpha+I \beta)-r_n( \alpha+I \beta) \|_{\mathbb{R}^6} &\leq & \frac{1}{2} \biggl( \sum_{|A|=0}^3\| F_A(\alpha+J \beta)-R_A(\alpha+J \beta) \|_{\mathbb{C}}+\\
&& + \sum_{|A|=0}^3\| F_A(\alpha-J \beta)-R_A(\alpha-J \beta) \|_{\mathbb{C}}+\\
&& +\sum_{|A|=0}^3\| F_A(\alpha+J \beta)-R_A(\alpha+J \beta) \|_{\mathbb{C}}+\\
&& +\sum_{|A|=0}^3\| F_A(\alpha-J \beta)-R_A(\alpha-J \beta) \|_{\mathbb{C}} \biggl)\\
& <& \! \! \! \! \varepsilon.
\end{eqnarray*}
Therefore we get \eqref{i2}.
\end{proof}
\begin{nb}
In this setting in the numerator of a rational function it is possible to find the zero divisors. However, this is not possible at the denominator because it is real by definition of rational function. Nevertheless, this is not a problem because the poles in A of Theorem \ref{th1} remain spheres or real points. The only thing that can change is the multiplicity of at most two points on the sphere (\cite{GPS}).
\end{nb}
From the previous theorem follows the implication $ 6) \Longrightarrow 2)$ of Theorem \ref{main}.
\\In the rest of the paper we will show many auxiliary results, which will help us to prove the last implication $ 3) \Longleftrightarrow 4)$ of Theorem \ref{main} (see Proposition \ref{res7}).
\\We introduce some notations which will be used for the next results. Thus, we define
$$ D^{+}= \{z \in D: \hbox{Im}(z) \geq 0\},$$
$$ D^{-}= \{z \in D: \hbox{Im}(z) \leq 0\},$$
$$ D_{\mathbb{R}}=D \cap \mathbb{R},$$
$$ D^{*}=D^{+} \setminus \mathbb{R}.$$
Let $ \partial D$ be the boundary of $D$. We define the real positive continuous function $h$ on $D_{\mathbb{R}}$ in the following way
$$ h(x)= \hbox{dist}(x, \partial D)= \inf_{z \in \partial D}|z-x|.$$
Moreover, we can set
$$ W= \{x+yi \in \mathbb{C}: x \in D_{\mathbb{R}}: 0 \leq y < h(x) \},$$
$$ W^{*}=W \setminus D_{\mathbb{R}}.$$
We observe that
$$ W= \{x+rh(x)i: \, x \in D_{\mathbb{R}}, \, r \in[0,1) \},$$
$$ W^*= \{x+rh(x)i: \, x \in D_{\mathbb{R}}, \, r \in (0,1) \},$$
$$ D_{\mathbb{R}}= \{x+rh(x)i: \, x \in D_{\mathbb{R}}, \, r=0 \}.$$
In order to obtain a precise description of the homology of $ \Omega_D^4$ in terms of the topology of $D$ (see Proposition \ref{desc}) it is crucial the following result.
\begin{prop}
\label{split3}
$$ H_{k}(\mathbb{S}_{\mathbb{R}_3})= \begin{cases}
\mathbb{Z} \qquad \qquad  k=0,4\\
\mathbb{Z} \oplus \mathbb{Z} \qquad k=2\\
0 \qquad \qquad k=1, 3, \, k \geq 5
\end{cases}
$$
\end{prop}
\begin{proof}
By Proposition \ref{split2} we know that $ \mathbb{S}_{\mathbb{R}_3}= \mathbb{S}_{\mathbb{R}_2} \times \mathbb{S}_{\mathbb{R}_2}$, so the formula follows by using the well known homology of a 2-sphere and the K\"unneth formula, which is without the torsion part since $\mathbb{Z}$ is a flat module.
\end{proof}
\begin{prop}
\label{desc}
Let $D$ be a symmetric open subset of $ \mathbb{C}$. We assume that the axially symmetric set $ \Omega_D^4$ is connected. Recalling the reduced homology $\widetilde{H}_0$ we have
\begin{equation}
\label{sec1}
0 \to H_1(D^+) \to H_5(\Omega_D^4) \to \widetilde{H}_0(D_{\mathbb{R}}) \to 0,
\end{equation}

\begin{equation}
\label{hom1}
H_{4}(\Omega_D^4)= \begin{cases}
0 \qquad D_{\mathbb{R}} \neq \emptyset \\
\mathbb{Z} \qquad D_{\mathbb{R}}= \emptyset,
\end{cases}
\end{equation}

\begin{equation}
\label{sec2}
0 \to H_{1}(D^+) \oplus H_1(D^+) \to H_3(\Omega_D^4) \to \widetilde{H}_0(D_{\mathbb{R}}) \oplus \widetilde{H}_0(D_{\mathbb{R}}) \to 0,
\end{equation}

\begin{equation}
\label{hom2}
 H_{2}(\Omega_D^4)= \begin{cases}
0 \qquad \qquad D_{\mathbb{R}} \neq \emptyset \\
\mathbb{Z} \oplus \mathbb{Z} \qquad D_{\mathbb{R}}= \emptyset,
\end{cases}
\end{equation}

\begin{equation}
\label{sec3}
0 \to H_1(D^+) \to H_1(\Omega_D^4) \to 0.
\end{equation}
\end{prop}
\begin{proof}
By the following equalities $ \Omega_D^4=\Omega^4_{D^*} \cup \Omega_W^4$ and $\Omega^4_{W^*}=\Omega^4_{D^*} \cap \Omega^4_W$ we can build  a Mayer-Vietoris sequence for homology
$$... \to H_{k+1}(\Omega_D^4) \to H_k(\Omega_{W^*}^4) \to H_k(\Omega_{D^*}^4) \oplus H_k(\Omega_W^4) \to H_k(\Omega_D^4) \to ... $$
Similar to the paper \cite[Prop. 2.5]{BW1} we have
$$ \Omega_{W^*}^4 \sim \mathbb{S}_{\mathbb{R}_3} \times D_{\mathbb{R}},$$
$$ \Omega_W^4 \sim D_{\mathbb{R}},$$
$$ \Omega_{D^*}^4 \sim \mathbb{S}_{\mathbb{R}_3} \times D^* \sim \mathbb{S}_{\mathbb{R}_3} \times D^+ .$$

Thus, we can write the Mayer-Vietoris sequence in the following way
\begin{equation}
... \to H_{k+1}(\Omega_D^4) \to H_k( \mathbb{S}_{\mathbb{R}_3} \times D_{\mathbb{R}}) \to H_k(\mathbb{S}_{\mathbb{R}_3} \times D^+) \oplus H_k( D_{\mathbb{R}}) \to H_k(\Omega_D^4) \to ...
\end{equation}
We known that the homology groups of the sphere $ \mathbb{S}_{\mathbb{R}_3}$ are torsion-free, so by the K\"unneth formula we have
$$ H_\ell(\mathbb{S}_{\mathbb{R}_3} \times X) \simeq \bigl( H_0(\mathbb{S}_{\mathbb{R}_3}) \otimes H_\ell(X) \bigl) \oplus \bigl( H_2(\mathbb{S}_{\mathbb{R}_3}) \otimes H_{\ell-2}(X) \bigl) \oplus \bigl( H_4(\mathbb{S}_{\mathbb{R}_3}) \otimes H_{\ell-4}(X) \bigl), \qquad \ell \geq 4,$$

$$  H_3(\mathbb{S}_{\mathbb{R}_3} \times X) \simeq \bigl( H_0(\mathbb{S}_{\mathbb{R}_3}) \otimes H_3(X) \bigl) \oplus \bigl( H_2(\mathbb{S}_{\mathbb{R}_3}) \otimes H_1(X) \bigl),$$

$$  H_2(\mathbb{S}_{\mathbb{R}_3} \times X) \simeq \bigl( H_0(\mathbb{S}_{\mathbb{R}_3}) \otimes H_2(X) \bigl) \oplus \bigl( H_2(\mathbb{S}_{\mathbb{R}_3}) \otimes H_0(X) \bigl),$$

$$ H_1(\mathbb{S}_{\mathbb{R}_3} \times X) \simeq H_0(\mathbb{S}_{\mathbb{R}_3}) \otimes H_1(X),$$

$$ H_0(\mathbb{S}_{\mathbb{R}_3} \times X) \simeq H_0(\mathbb{S}_{\mathbb{R}_3}) \otimes H_0(X).$$
Hence
\begin{eqnarray}
\nonumber
\label{wseq1}
 ... \! \! \! &\to& \! \! \! H_{k+1}(\Omega_D^4)   \to \bigl( H_0(\mathbb{S}_{\mathbb{R}_3}) \otimes H_k(D_{\mathbb{R}}) \bigl) \oplus \bigl( H_2(\mathbb{S}_{\mathbb{R}_3}) \otimes H_{k-2}(D_{\mathbb{R}}) \bigl) \oplus \bigl( H_4(\mathbb{S}_{\mathbb{R}_3}) \otimes H_{k-4}(D_{\mathbb{R}}) \bigl) \to \\ \nonumber
&\to& \! \! \! \!  \bigl( H_0(\mathbb{S}_{\mathbb{R}_3}) \otimes H_k(D^+) \bigl) \oplus \bigl( H_2(\mathbb{S}_{\mathbb{R}_3}) \otimes H_{k-2}(D^+) \bigl) \oplus \bigl( H_4(\mathbb{S}_{\mathbb{R}_3}) \otimes H_{k-4}(D^+) \bigl) \oplus H_k( D_{\mathbb{R}}) \to\\ 
&\to& H_{k}(\Omega_D^4) \to...
\end{eqnarray}
Now we observe that $H_k(D_{\mathbb{R}})= \{0\}$ for $k>0$ and $H_k(D^+)=\{0\}$ for $k>1$. Putting these in the long exact sequence \eqref{wseq1} we get
\begin{eqnarray*}
\nonumber
0  \! \! \! &\to& \! \! \!  \!   H_{4}(\mathbb{S}_{\mathbb{R}_3}) \otimes H_1(D^+) \to H_5(\Omega_D^4) \to H_4(\mathbb{S}_{\mathbb{R}_3}) \otimes H_0(D_{\mathbb{R}}) \to H_4( \mathbb{S}_{\mathbb{R}_3}) \otimes H_{0}(D^+) \to H_{4}(\Omega_D^4) \to\\ \nonumber
& \to&  \! \! \!  \! 0 \to H_{2}(\mathbb{S}_{\mathbb{R}_3}) \otimes H_1(D^+) \to H_3(\Omega_D^4) \to H_2(\mathbb{S}_{\mathbb{R}_3}) \otimes H_0(D_{\mathbb{R}}) \to H_2(\mathbb{S}_{\mathbb{R}_3}) \otimes H_0(D^+) \to H_2(\Omega_D^4) \to \\ \nonumber
& \to& \! \! \!  \! 0 \to H_0(\mathbb{S}_{\mathbb{R}_3}) \otimes H_1(D^+) \to H_1(\Omega_D^4) \to H_0(\mathbb{S}_{\mathbb{R}_3}) \otimes H_0(D_{\mathbb{R}}) \to  \bigl( H_0(\mathbb{S}_{\mathbb{R}_3}) \otimes H_0(D^+) \bigl) \oplus H_0(D_{\mathbb{R}}) \to \\
& \to& \! \! \!  \! H_0(\Omega_D^4) \to 0.
\end{eqnarray*}
This allows us to split the sequence in the following way
\begin{eqnarray}
\label{seq4}
\nonumber
 0 &\to& \! \! \! \! \!  H_{4}(\mathbb{S}_{\mathbb{R}_3}) \otimes H_1(D^+) \to H_5(\Omega_D^4) \to H_4(\mathbb{S}_{\mathbb{R}_3}) \otimes H_0(D_{\mathbb{R}}) \to \\
& \to & \! \! \! \! \!  H_4( \mathbb{S}_{\mathbb{R}_3}) \otimes H_{0}(D^+) \to H_{4}(\Omega_D^4) \to 0
\end{eqnarray}

\begin{eqnarray}
\label{seq5}
\nonumber
0 &\to& \! \! \! \! \! H_{2}(\mathbb{S}_{\mathbb{R}_3}) \otimes H_1(D^+) \to H_3(\Omega_D^4) \to H_2(\mathbb{S}_{\mathbb{R}_3}) \otimes H_0(D_{\mathbb{R}}) \to \\
&\to & \! \! \! \! \!H_2(\mathbb{S}_{\mathbb{R}_3}) \otimes H_0(D^+) \to  H_2(\Omega_D^4) \to 0
\end{eqnarray}

\begin{eqnarray}
\label{seq6}
\nonumber
 0 &\to& \! \! \! \! \! H_0(\mathbb{S}_{\mathbb{R}_3}) \otimes H_1(D^+) \to H_1(\Omega_D^4) \to  H_0(\mathbb{S}_{\mathbb{R}_3}) \otimes H_0(D_{\mathbb{R}}) \to  \\
& \to&  \! \! \! \! \! \bigl( H_0(\mathbb{S}_{\mathbb{R}_3}) \otimes H_0(D^+) \bigl) \oplus H_0(D_{\mathbb{R}}) \to  H_0(\Omega_D^4) \to 0.
\end{eqnarray}
\fbox{\emph{First case}: $D_{\mathbb{R}} \neq \emptyset$ }
\newline
\\Since $D_{\mathbb{R}} \hookrightarrow D^+$, we have a surjective group homomorphism $H_0(D_{\mathbb{R}}) \to H_0(D^+)$ and by definition of reduced homology this has kernel $ \widetilde{H}_0(D_{\mathbb{R}})$. Let us define in \eqref{seq4} the following homomorphism $ \alpha: H_4(\mathbb{S}_{\mathbb{R}_3}) \otimes H_0(D_{\mathbb{R}}) \to H_4( \mathbb{S}_{\mathbb{R}_3}) \otimes H_{0}(D^+)$. Therefore, we can split the sequence \eqref{seq4} in the following way
\begin{equation}
\label{wseq2}
0 \to H_4(\mathbb{S}_{\mathbb{R}_3}) \otimes H_1(D^+) \to H_5(\Omega_D^4) \to \hbox{ker} (\alpha) \to 0,
\end{equation}
and
\begin{equation}
\label{wseq3}
0 \to (H_4(\mathbb{S}_{\mathbb{R}_3}) \otimes H_0(D_{\mathbb{R}}))/ \hbox{ker} (\alpha) \overset{\alpha}{\to} H_4(\mathbb{S}_{\mathbb{R}_3}) \otimes H_{0}(D^+) \to H_4(\Omega_D^4) \to 0.
\end{equation}
Since $H_4(\mathbb{S}_{\mathbb{R}_3}) \otimes H_1(D^+) \simeq H_1(D^+)$ and $\hbox{ker} (\alpha) \simeq \widetilde{H}_0(D_{\mathbb{R}})$, by \eqref{wseq2} we have \eqref{sec1}. Moreover, the exact sequence \eqref{wseq3} implies $H_{4}(\Omega_D^4) \simeq \{0\}$, since $ \alpha$ is surjective.
\\ Now, we observe that by Proposition \ref{split3} and by the distributive property of the tensor product we can write
$$ H_{2}(\mathbb{S}_{\mathbb{R}_3}) \otimes H_0(D_{\mathbb{R}}) \simeq (\mathbb{Z} \oplus \mathbb{Z}) \otimes H_0(D_{\mathbb{R}}) \simeq ( \mathbb{Z} \otimes H_{0}(D_{\mathbb{R}})) \oplus (\mathbb{Z} \otimes H_0(D_{\mathbb{R}})) \simeq H_{0}(D_{\mathbb{R}}) \oplus H_{0}(D_{\mathbb{R}}).$$
Then, we can write the exact sequence \eqref{seq5} as
\begin{eqnarray}
\label{wseq0}
0 &\to& \! \! \! \!H_1(D^{+}) \otimes H_1(D^+) \to H_3(\Omega_D^4) \to  H_0(D_{\mathbb{R}}) \oplus H_{0}(D_{\mathbb{R}})  \to\\ \nonumber
&\to& \! \! \! \! H_0(D^+) \oplus H_0(D^+) \to H_2(\Omega_D^4) \to 0.
\end{eqnarray}
In this case the inclusion map $ D_{\mathbb{R}} \hookrightarrow D^+$ yields a surjective group morphism
$\beta: H_0({D_{\mathbb{R}}}) \oplus H_0(D_{\mathbb{R}}) \to H_0(D^+) \oplus H_0(D^+)$ with $ \hbox{ker} (\beta) \simeq \widetilde{H}_0(D_{\mathbb{R}})  \oplus \widetilde{H}_0(D_{\mathbb{R}})$.
Thus we can split the sequence as
\begin{equation}
\label{wseq7}
0 \to H_1(D^+) \oplus H_1(D^+) \to H_3(\Omega_D^4) \to \hbox{ker} (\beta) \to 0,
\end{equation}
and
\begin{equation}
\label{wseq8}
0 \to (H_0(D_{\mathbb{R}}) \oplus H_0(D_{\mathbb{R}}))/ \hbox{ker} (\beta) \overset{\beta}{\to} H_0(D^+) \oplus H_0(D^+) \to H_2(\Omega_D^4) \to 0.
\end{equation}
Hence, from \eqref{wseq7} we have \eqref{sec2}. Furthermore since $\beta$ is surjective we obtain that $H_2(\Omega_D^4) \simeq \{0\}.$
\newline
\newline
\fbox{\emph{Second case} $D_{\mathbb{R}}= \emptyset$}
\newline
\newline
It is obvious that $ H_0(D_{\mathbb{R}}) \simeq \{0\}$. From \eqref{seq4} we obtain
\begin{equation}
\label{wseq9}
0 \to H_4(\mathbb{S}_{\mathbb{R}_3}) \otimes H_1(D^+) \to H_5(\Omega_D^4) \to 0
\end{equation}
and
\begin{equation}
\label{wseq10}
0 \to H_4(\mathbb{S}_{\mathbb{R}_3}) \otimes H_0(D^+) \to H_4(\Omega_D^4) \to 0.
\end{equation}
Since $H_4(\mathbb{S}_{\mathbb{R}_3}) \otimes H_1(D^+) \simeq H_1(D^+)$ by the sequence \eqref{wseq9} we get \eqref{sec1}. Moreover, by the fact that $H_4(\mathbb{S}_{\mathbb{R}_3}) \otimes H_0(D^+) \simeq \mathbb{Z}$ (since $D^+$ is connected) we have $H_{4}(\Omega_D^4) \simeq \mathbb{Z}.$
\\On the other hand from \eqref{seq5} we have
\begin{equation}
\label{wseq11}
0 \to H_2(\mathbb{S}_3) \otimes H_1(D^+) \to H_3(\Omega_D^4) \to 0,
\end{equation}
and
\begin{equation}
\label{wseq12}
0 \to H_2(\mathbb{S}_{\mathbb{R}_3}) \otimes H_0(D^+) \to H_2(\Omega_D^4) \to 0.
\end{equation}
By $H_2(\mathbb{S}_3) \otimes H_1(D^+) \simeq  H_1(D^+) \oplus H_1(D^+)$ we get \eqref{sec2}. By \eqref{wseq12} and the fact that $D^{+}$ is connected we have
$$ H_2(\Omega_D^4) \simeq  H_2(\mathbb{S}_{\mathbb{R}_3}) \otimes H_0(D^+) \simeq (\mathbb{Z} \otimes H_0(D^+)) \oplus (\mathbb{Z} \otimes H_0(D^+)) \simeq (\mathbb{Z} \otimes \mathbb{Z}) \oplus (\mathbb{Z} \otimes \mathbb{Z}) \simeq \mathbb{Z} \oplus \mathbb{Z}.$$
\newline
Finally, we have to prove \eqref{sec3}. By Proposition \ref{split3} we can write the sequence \eqref{seq6} in the following way
\begin{equation}
0 \to H_1(D^+) \to H_1(\Omega_D^4) \to H_0(D_{\mathbb{R}}) \to H_0(D^+) \oplus H_0(D_{\mathbb{R}}) \to H_0(\Omega_D^4) \to 0.
\end{equation}
Since the map $ H_0(D_{\mathbb{R}}) \to H_0(D^+) \oplus H_0(D_{\mathbb{R}})$ is injective by the properties of exact sequence we obtain
$$ 0 \to H_1(D^+) \to H_1(\Omega_D^4) \to 0,$$
which is  exactly \eqref{sec3}.
\end{proof}

\begin{Cor}
Let $D$ be a symmetric open subset of $ \mathbb{C}$. We assume that the corresponding axially symmetric domain $ \Omega_D^4$ is connected. Moreover, let us assume that $D$ is a bounded domain with smooth boundary. Then all the homology groups are finitely generated  and Proposition \ref{desc} implies the following description of the Betti numbers $b_k= \hbox{dim} \, H_k(\, \, \, , \mathbb{Z}) \otimes_{\mathbb{Z}} \mathbb{R}$. Let
$$ r:= \begin{cases}
b_0(D_{\mathbb{R}})-1 \qquad D_{\mathbb{R}} \neq \emptyset\\
0 \qquad \qquad \qquad D_{\mathbb{R}}= \emptyset.
\end{cases}
$$
Then
\begin{equation}
\label{n1}
b_1(\Omega_D^4)= \frac{1}{2}(b_1(D)-r),
\end{equation}
\begin{equation}
\label{n2}
b_2(\Omega_D^4)=\begin{cases}
	2 \qquad D_{\mathbb{R}}= \emptyset\\
	0 \qquad D_{\mathbb{R}} \neq \emptyset,
\end{cases}
\end{equation}
\begin{equation}
\label{n3}
b_3(\Omega_D^4)=b_1(D)+r,
\end{equation}
\begin{equation}
\label{n4}
b_4(\Omega_D^4)=\begin{cases}
	1 \qquad D_{\mathbb{R}}= \emptyset\\
	0 \qquad D_{\mathbb{R}} \neq \emptyset,
\end{cases}
\end{equation}
\begin{equation}
\label{n5}
b_5(\Omega_D^4)= \frac{1}{2}(b_1(D)+r).
\end{equation}
\end{Cor}

\begin{proof}
The formulas \eqref{n1}, \eqref{n2}, \eqref{n3}, \eqref{n4}, \eqref{n5} follow respectively by \eqref{sec3}, \eqref{hom2}, \eqref{sec2}, \eqref{hom1}, \eqref{sec1}.
\end{proof}

\begin{nb}
By Proposition \ref{np} we have 
$$ b_1(\Omega_D^4)=b_1(\Omega_D^2(\mathbb{R}_2)),$$
$$ b_2(\Omega_{D}^4)= b_2(\Omega_D^2(\mathbb{R}_2)) + b_2(\Omega_D^2(\mathbb{R}_2)),$$
$$ b_3(\Omega_{D}^4)= b_3(\Omega_D^2(\mathbb{R}_2)) + b_3(\Omega_D^2(\mathbb{R}_2)),$$
$$ b_4(\Omega_D^4)=b_2(\Omega_{D}^2(\mathbb{R}_2)),$$
$$ b_5(\Omega_D^4)=b_3(\Omega_{D}^2(\mathbb{R}_2)).$$
\end{nb}

\begin{Cor}
\label{result1}
Let $D$ be a symmetric open subset of $ \mathbb{C}$. Let us assume that $ \Omega_D^4$ is not necessarily connected. Then
\begin{equation}
\label{hom1b}
H_{4}(\Omega_D^4)= \begin{cases}
0  \, \, \, \, \qquad D_{\mathbb{R}} \neq \emptyset \\
\mathbb{Z}^k \qquad D_{\mathbb{R}}= \emptyset,
\end{cases}
\end{equation}

\begin{equation}
\label{hom2b}
H_{2}(\Omega_D^4)= \begin{cases}
0 \qquad \qquad \quad D_{\mathbb{R}} \neq \emptyset \\
\mathbb{Z}^k \oplus \mathbb{Z}^k \qquad D_{\mathbb{R}}= \emptyset,
\end{cases}
\end{equation}
where $k$ denotes the number of connected components of $D^+$ which do not intersect $ \mathbb{R}$.
\\Let $ \widehat{H}_0(D_{\mathbb{R}})$ be the kernel of the homomorphism $i_*: H_{0}(D_{\mathbb{R}}) \to H_0(D^+)$. Then we have the following exact sequences
\begin{equation}
\label{sec1b}
0 \to H_1(D^+) \to H_5(\Omega_D^4) \to \widehat{H}_0(D_{\mathbb{R}}) \to 0,
\end{equation}

\begin{equation}
\label{sec2b}
0 \to H_{1}(D^+) \oplus H_1(D^+) \to H_3(\Omega_D^4) \to \widehat{H}_0(D_{\mathbb{R}}) \oplus \widehat{H}_0(D_{\mathbb{R }}) \to 0,
\end{equation}

\begin{equation}
\label{sec3b}
0 \to H_1(D^+) \to H_1(\Omega_D^4) \to 0.
\end{equation}
\end{Cor}
\begin{proof}
This is a consequence of Proposition \ref{desc} and the fact that the homology of a disconnected space is isomorphic to the direct sum of the homology of its connected components.
\end{proof}

Now we explain the geometric meaning of the short exact sequence \eqref{sec1}. If we consider an element $ \alpha \in H_1(D^+)$ we can represent it as a finite formal $\mathbb{Z}$-linear combination of closed curves $\gamma_j: S^1 \to D^+$. Each of them defines a map $ \eta:S^1 \times \mathbb{S}_{\mathbb{R}_3} \to \Omega_{D}^4$ through
$$ \eta(t,I)= \hbox{Re}(\gamma_j(t))+I \cdot \hbox{Im}(\gamma_j(t)).$$
The fundamental class of the real five-dimensional manifold $\eta(S^1 \times \mathbb{S}_{\mathbb{R}_3})$ defines an element in $H_5(\Omega_{D}^4)$.

It is also possible to prove that the sequence \eqref{sec1} has not a natural splitting. Given an element $ \beta \in H_0(D_{\mathbb{R}})$ we can represent it as a formal $ \mathbb{Z}$-linear combinations of points $ \sum n_i \{p_i\}$. Let us assume that $\beta$ is in the kernel of the natural map to $\mathbb{Z}$ which is given by $ \sum n_i \{p_i\} \mapsto \sum n_i$. This implies that $\beta$ is the sum of elements of the form $1 \{p_i\}-1\{q_i\}$. Now, we can choose a curve $\gamma: [0,1] \to D^+$ such that $\gamma(0)=p_i$, $\gamma(1)=q_i$ and $\gamma(t) \in D^+ \setminus \mathbb{R}$ for $0<t<1$. Then $\Omega_{\gamma([0,1])}^4$ is a 5-sphere defining an element in $H_5(\Omega_{D}^4)$. However, we observe that the construction depends on the choice of the curve $\gamma$. This means that the sequence \eqref{sec1} has not a natural splitting.
\newline
\newline
It is possible to have a geometric meaning also for the exact sequence \eqref{sec2}. If we consider a couple of element $( \alpha, \alpha) \in H_1(D^+) \oplus H_1(D^+)$ we can represent it as a couple of finite $ \mathbb{Z}$-linear combination of closed curves $ \gamma_j:S^{1} \to D^{+}$. This couple of curves defines a map $ \eta': S^1  \times \mathbb{S}^2 \to \Omega^4_D$ through
$$ \eta'(t,K)=\hbox{Re}(\gamma_j(t))+K \cdot \hbox{Im}(\gamma_j(t)),$$
where $K \in \mathbb{S}^2 \subset \mathbb{S}_{\mathbb{R}_3}$
The fundamental class of $\eta' \left( S^1 \times \mathbb{S}^2 \right)$ defines an element in $H_3(\Omega_{D}^4)$.

As before it is possible to prove that the sequence \eqref{sec2} has not a natural splitting. Given a couple $ (\beta, \beta) \in H_0(D_{\mathbb{R}}) \oplus H_0(D_{\mathbb{R}})$ we can represent each $\beta$ as a formal $ \mathbb{Z}$-linear combination of points $ \sum n_i \{p_i\}$. Let us assume that each $\beta$ is in the kernel of the natural map to $\mathbb{Z}$ which is given by $ \sum n_i \{p_i\} \mapsto \sum n_i$. This implies that $\beta$ is the sum of elements of the form $1 \{p_i\}-1\{q_i\}$. Now, we can choose a curve $\gamma: [0,1] \to D^+$ such that $\gamma(0)=p_i$, $\gamma(1)=q_i$ and $\gamma(t) \in D^+ \setminus \mathbb{R}$ for $0<t<1$. Then $\Omega^2_{\gamma([0,1])}$ (see Definition \eqref{ax}) is a 3-sphere defining an element in $H_3(\Omega_{D}^4)$. However, we observe that the construction depends on the choice of the curve $\gamma$. This means that also the sequence \eqref{sec2} has not a natural splitting.
\newline
\newline
We recall from \cite[Lemma 2.9]{BW1} and \cite[Cor. 2.10]{BW1} the following results.
\begin{lemma}
\label{res2}
Let $D \subset \mathbb{C}$ be a symmetric open subset. Then, there is a natural exact sequence
\begin{equation}
0 \to H_1(D^+) \oplus H_1(D^-) \to H_1(D) \to \widehat{H}_0(D_\mathbb{R}) \to 0
\end{equation}
\end{lemma}

\begin{Cor}
\label{res3}
Let $D \subset D_1$ be symmetric open subsets in $ \mathbb{C}$. Assume that $H_1(D) \to H_1(D_1)$ is injective. Then $H_1(D^+) \to H_1(D^+_1)$ is injective.
\end{Cor}

\begin{prop}
\label{res21}
Let $D$ be a symmetric open subset of $ \mathbb{C}$. Then we have the following exact sequence
\begin{equation}
0 \to H_1(D^+) \overset{\mathfrak{a}}{\to} H_1(D) \overset{\mathfrak{b}}{\to} H_5(\Omega_D^4) \to 0,
\end{equation}
where $ \mathfrak{a}$, $\mathfrak{b}$ are defined as follows. Let $ \tau: \mathbb{C} \to \mathbb{C}$ be the complex conjugate and let $ \xi: D \times \mathbb{S}_{\mathbb{R}_3} \to \Omega_D^4$ be the map given by
$$ \xi(x+ y i, J)=x+y J.$$
Thus, we define $ \mathfrak{a}(\gamma)=\gamma- \tau_{*} \gamma$ and $ \mathfrak{b}(\gamma)= \xi_*(\gamma \times [\mathbb{S}_{\mathbb{R}_3}])$, where $[\mathbb{S}_{\mathbb{R}_3}] \in H_4 (\mathbb{S}_{\mathbb{R}_3})$ denotes the fundamental class of $\mathbb{S}_{\mathbb{R}_3}$.
\end{prop}

\begin{proof}
We can assume that $D^+$ is connected, and hence $\Omega_D^4$ is connected too.
We cover $D^+$ by the two open subsets $D^*$ and $W$, as in the proof of Proposition \ref{desc}. Now we consider
$$ V= \{z \in \mathbb{C}: \, z \in W \, \, \hbox{or} \, \, \bar{z} \in W\}.$$
As in \cite[Prop. 2.11]{BW1} we have the following coverings of $D$, $D \times \mathbb{S}_{\mathbb{R}_3}$ and $ \Omega_D^4$:
$$ D=(D \setminus D_{\mathbb{R}}) \cup V,$$
$$ D \times \mathbb{S}_{\mathbb{R}_{3}}= \bigl( (D \setminus D_{\mathbb{R}})\times \mathbb{S}_{\mathbb{R}_3} \bigl) \, \, \cup  \, \, (V \times \mathbb{S}_{\mathbb{R}_3}), $$
$$ \Omega_D^4=\Omega_{D^*}^4 \cup \Omega_W^4.$$
Moreover,
$$ \Omega_{W^*}^4= \Omega_{D^*}^4 \cap \Omega_W^4.$$
Via the map $\xi$ defined in the hypothesis we get a morphism between the Mayer-Vietoris sequences obtained from the previous coverings:
$$
\small
\begin{tikzcd}
... \arrow[r, ""]
& H_k\bigl( (V \setminus D_{\mathbb{R}}) \times \mathbb{S}_{\mathbb{R}_3} \bigl) \arrow[r, ""] \arrow[d, ""']
& H_k((D \setminus D_{\mathbb{R}}) \times \mathbb{S}_{\mathbb{R}_3}) \oplus H_k(V \times \mathbb{S}_{\mathbb{R}_3}) \arrow[r, ""] \arrow[d, ""]
& H_k(D \times \mathbb{S}_{\mathbb{R}_3}) \arrow[r, ""] \arrow[d, ""]
& ... \\
...\arrow[r, ""] & H_k(\Omega^4_{W^{*}}) \arrow[r, ""] & H_k(\Omega^4_{D^*}) \oplus H_k(\Omega^4_W) \arrow[r, ""] & H_k(\Omega^4_D) \arrow[r, ""] & ...
\end{tikzcd}
$$
In particular, we get
$$
\small
\begin{tikzcd}
H_5\bigl( (V \setminus D_{\mathbb{R}}) \times \mathbb{S}_{\mathbb{R}_3} \bigl) \arrow[r, ""] \arrow[d, ""']
& H_5((D \setminus D_{\mathbb{R}}) \times \mathbb{S}_{\mathbb{R}_3}) \oplus H_5(V \times \mathbb{S}_{\mathbb{R}_3}) \arrow[r, ""] \arrow[d, ""]
& H_5(D \times \mathbb{S}_{\mathbb{R}_3}) \arrow[r, ""] \arrow[d, ""]
& C \arrow[r, ""] \arrow[d, ""]
& 0 \arrow[d, ""]\\
H_5(\Omega^4_{W^{*}}) \arrow[r, ""] & H_5(\Omega^4_{D^*}) \oplus H_5(\Omega^4_W) \arrow[r, ""] & H_5(\Omega^4_D) \arrow[r, ""] & C' \arrow[r, ""] & 0
\end{tikzcd}
$$

where
$$ C= \hbox{ker} \bigl[H_4\bigl( (V \setminus D_{\mathbb{R}}) \times \mathbb{S}_{\mathbb{R}_3} \bigl) \to H_4((D \setminus D_{\mathbb{R}}) \times \mathbb{S}_{\mathbb{R}_3}) \oplus H_4(V \times \mathbb{S}_{\mathbb{R}_3}) \bigl]
$$
and
$$ C'= \hbox{ker} \bigl[ H_4(\Omega^4_{W^{*}}) \to H_4(\Omega^4_{D^*}) \oplus H_4(\Omega^4_W) \bigl].$$
Now, let us consider a domain $M \subset \mathbb{C}$. We recall that $H_{\ell}(M)=0$ for $ \ell \geq 2$. Therefore, by the K\"unneth formula we obtain
\begin{equation}
\label{hom3}
H_4(M \times \mathbb{S}_{\mathbb{R}_3}) \simeq H_{0}(M),
\end{equation}

\begin{equation}
\label{hom4}
H_{5}(M \times \mathbb{S}_{\mathbb{R}_3}) \simeq H_1(M).
\end{equation}
Now, we remark that $V \setminus D_{\mathbb{R}}$ is the disjoint union of two open subsets, i.e.
\begin{equation}
\label{n9}
V \setminus D_{\mathbb{R}}=\left(D^+ \cap (V \setminus D_{\mathbb{R}}) \right) \, \, \bigsqcup \, \, \left(D^- \cap (V \setminus D_{\mathbb{R}}) \right),
\end{equation}
where the two open sets are homotopic to $D_{\mathbb{R}}$ and $V$ is homotopic to $D_{\mathbb{R}}$. By \eqref{hom3} we have
$$ C \simeq \hbox{ker}[H_{0}(V \setminus D_{\mathbb{R}}) \to  H_0(D \setminus D_{\mathbb{R}})  \oplus H_0(V) ].$$
As proved in \cite[Prop. 2.11]{BW1} we have
$$ H_{0}(D_{\mathbb{R}}) \simeq \hbox{ker}[H_0(V \setminus D_{\mathbb{R}}) \to H_0(V)].$$
Therefore, by definition of reduced homology we have
 $$ C \simeq \widetilde{H}_0(D_{\mathbb{R}}).$$
Due to the following homotopy equivalences (see \cite[Prop. 2.5]{BW1}):
\begin{equation}
\label{n6}
\Omega_{W^*}^4 \simeq D_{\mathbb{R}} \times \mathbb{S}_{\mathbb{R}_3}, \qquad \Omega_{D^*}^4 \simeq D^+ \times \mathbb{S}_{\mathbb{R}_3}, \qquad \Omega_W^4 \simeq D_{\mathbb{R}},
\end{equation}
the formula \eqref{hom3} and the fact that $H_{k}(D_{\mathbb{R}})={0}$ for $k>0$ we have:
\begin{eqnarray*}
C'&=& \hbox{ker}[H_4(\Omega^4_{W^*}) \to H_4(\Omega^4_{D^*}) \oplus H_4(\Omega^4_W)]\\
&\simeq & \hbox{ker}[H_0(D_{\mathbb{R}}) \to H_0(D^+) \oplus H_4(D_{\mathbb{R}})]\\
& \simeq & \hbox{ker}[H_0(D_{\mathbb{R}}) \to H_0(D^+)]\\
& \simeq & \widetilde{H}_0(D_{\mathbb{R}}).
\end{eqnarray*}
Therefore
$$ C' \simeq \widetilde{H}_0(D_{\mathbb{R}}).$$
Hence by \eqref{hom4} and the previous homotopy equivalences we have $H_5(\Omega^4_{D^*}) \simeq H_1(D^+)$ and $ H_5(\Omega^4_W)=0$. Moreover by \eqref{n9} and we get
\begin{eqnarray*}
H_{5}((V\setminus D_{\mathbb{R}}) \times \mathbb{S}_{\mathbb{R}_3}) &\simeq& H_1(V \setminus D_{\mathbb{R}}) \simeq H_{1}\left(D^+ \cap (V \setminus D_{\mathbb{R}}) \right) \oplus H_{1}\left(D^- \cap (V \setminus D_{\mathbb{R}}) \right) \\
&\simeq& H_{1}(D_{\mathbb{R}}) \oplus H_{1}(D_{\mathbb{R}}) \simeq 0.
\end{eqnarray*}

Combining these facts we obtain the following commutative diagram

$$
\begin{tikzcd}
0 \arrow[r, ""] \arrow[d, ""]
& H_1(D^+) \oplus H_1(D^-) \arrow[r] \arrow[d]
& H_1(D) \arrow[r] \arrow[d]
&\widetilde{H}_0(D_{\mathbb{R}}) \arrow[r, ""] \arrow[d]
& 0 \ \arrow[d, ""]\\
0 \arrow[r, ""] & H_1(D^+) \arrow[r] & H_5(\Omega^4_D) \arrow[r] & \widetilde{H}_0(D_{\mathbb{R}}) \arrow[r, ""] & 0
\end{tikzcd}
$$
By similar computations of \cite[Prop.2.11]{BW1} we have the thesis.
\end{proof}

\begin{prop}
\label{res2b}
Let $D$ be a symmetric open subset of $ \mathbb{C}$. Then there is a natural exact sequence
\begin{equation}
\label{wseq15}
0 \to H_1(D^+) \oplus H_1(D^+) \overset{\alpha'}{\to} H_1(D) \oplus H_1(D) \overset{\beta'}{\to} H_3(\Omega_D^4) \to 0.
\end{equation}
Let $ \tau$ be the complex conjugation on $ \mathbb{C}$ and let $ \zeta: D \times \mathbb{S}^2  \to \Omega_D^4$ defined by
$$ \zeta (x+yi, I)=x+yI.$$
We observe that $I \in \mathbb{S}^2 \subset \mathbb{S}_{\mathbb{R}_3}$.
Thus, we can define $ \alpha'$ and $ \beta'$ as
$$ \alpha'(\gamma, \gamma_1)=\gamma - \tau_*(\gamma)+ \gamma_1- \tau_*(\gamma_1) =  \alpha(\gamma)+ \alpha(\gamma_1) $$
and
$$ \beta'(\gamma, \gamma_1)= (\zeta_{*}(\gamma), \zeta_{*}(\gamma_1)) \times [\mathbb{S}^2]= \left(\zeta_{*}(\gamma) \times [\mathbb{S}^2], \zeta_{*}(\gamma_1) \times [\mathbb{S}^2]\right).$$
where the map $ \alpha$ is defined in \cite[Prop. 2.11]{BW1} and
$[\mathbb{S}^2] \in H_2(\mathbb{S}^2)$ is the fundamental class of $ \mathbb{S}^2$.
\end{prop}
\begin{proof}
As in the previous proposition we assume that $D^+$ is connected, hence also $\Omega_D^4$ is connected. As in \cite[Prop. 2.11]{BW1} we can cover $D \times \mathbb{S}^2$ in the following way
$$ D \times \mathbb{S}^2= \bigl( (D \setminus D_{\mathbb{R}})\times \mathbb{S}^2 \bigl) \, \, \cup  \, \, (V \times \mathbb{S}^2),$$
where $V= \{z \in \mathbb{C}: \, z \in W \, \, \hbox{or} \, \, \bar{z} \in W\}$. Moreover we recall that
$$ D=(D \setminus D_{\mathbb{R}}) \cup V ,$$
and
$$ \Omega_D^4=\Omega_{D^*}^4 \cup \Omega_W^4.$$	
Furthermore,
$$ \Omega_{W^*}^4= \Omega_{D^*}^4 \cap \Omega_{W}^4.$$
By the map  $ \zeta: D \times \mathbb{S}^2 \to \Omega_D^4$ given by
$$ \zeta (x+yi, I)=x+yI$$
we get the following morphism between the respective Mayer-Vietoris sequences	

$$
\begin{tikzcd}
... \arrow[r, ""]
& A \oplus A \arrow[r, ""] \arrow[d, ""']
& B  \oplus  C \oplus B \oplus C \arrow[r, ""] \arrow[d, ""]
& E \oplus E \arrow[r] \arrow[d]
& ...\\
...\arrow[r, ""] & H_k(\Omega^4_{W^{*}}) \arrow[r, ""] & H_k(\Omega_{D^*}^4) \oplus H_k(\Omega^4_W) \arrow[r, ""] & H_k(\Omega^4_{D}) \arrow[r, ""] & ...
\end{tikzcd}
$$

where
$$ A:= H_k\bigl( (V \setminus D_{\mathbb{R}}) \times \mathbb{S}^2 \bigl),$$
$$ B:=H_k((D \setminus D_{\mathbb{R}}) \times \mathbb{S}^2),$$
$$C:=H_k(V \times \mathbb{S}^2),$$
$$E:=H_k\bigl( D \times \mathbb{S}^2 \bigl).$$

In particular we obtain
$$
\begin{tikzcd}
A' \oplus A' \arrow[r, ""] \arrow[d, ""']
& B' \oplus C' \oplus B' \oplus C' \arrow[r, ""] \arrow[d, ""']
& E' \oplus E' \arrow[r, ""] \arrow[d, ""']
& F \arrow[r, ""] \arrow[d, ""']
& 0  \arrow[d, ""'] \\
H_3(\Omega^4_{W^{*}}) \arrow[r, ""] & H_3(\Omega^4_{D^*}) \oplus H_3(\Omega^4_W) \arrow[r, ""] & H_3(\Omega^4_D) \arrow[r, ""] & F' \arrow[r, ""] & 0
\end{tikzcd}
$$

where
$$ A':= H_3\bigl( (V \setminus D_{\mathbb{R}}) \times \mathbb{S}^2 \bigl),$$
$$ B':=H_3((D \setminus D_{\mathbb{R}}) \times \mathbb{S}^2),$$
$$C':=H_3(V \times \mathbb{S}^2),$$
$$E':=H_3\bigl( D \times \mathbb{S}^2 \bigl),$$
and
\begin{eqnarray*}
F &=& \hbox{ker}  \bigl[H_2\bigl( (V \setminus D_{\mathbb{R}}) \times \mathbb{S}^2 \bigl) \oplus H_2\bigl( (V \setminus D_{\mathbb{R}}) \times \mathbb{S}^2 \bigl) \to H_2((D \setminus D_{\mathbb{R}}) \times \mathbb{S}^2) \oplus H_2(V \times \mathbb{S}^2) \oplus\\
&& \oplus H_2((D \setminus D_{\mathbb{R}}) \times \mathbb{S}^2) \oplus H_2(V \times \mathbb{S}^2)\bigl],
\end{eqnarray*}
$$ F'= \hbox{ker} \bigl[ H_2(\Omega^4_{W^{*}}) \to H_2(\Omega^4_{D^*}) \oplus H_2(\Omega^4_W) \bigl].$$
Now, for any domain $M \subset \mathbb{C}$, by the K\"unneth formula and dimensional reasons we have

\begin{equation}
\label{hom7}
H_2(M \times \mathbb{S}^2)= H_0(M).
\end{equation}

We recall that
\begin{equation}
\label{n10}
V \setminus D_{\mathbb{R}}=\bigl(D^+ \cap (V \setminus D_{\mathbb{R}}) \bigl) \, \, \bigsqcup \, \, (D^- \cap (V \setminus D_{\mathbb{R}}) \bigl),
\end{equation}
where the two open sets are homotopic to $D_{\mathbb{R}}$ and $V$ is homotopic to $D_{\mathbb{R}}$.
Hence by \eqref{hom7}
$$ F \simeq \hbox{ker}[H_0(V \setminus D_{\mathbb{R}}) \oplus H_0(V \setminus D_{\mathbb{R}}) \to H_0(D \setminus D_{\mathbb{R}})\oplus H_0(D \setminus D_{\mathbb{R}}) \oplus H_0(V) \oplus H_0(V)].$$
Now, we want to prove that
\begin{equation}
\label{hom8}
H_0(D_{\mathbb{R}}) \oplus H_0(D_{\mathbb{R}}) \simeq \hbox{ker}[H_0(V \setminus D_\mathbb{R}) \oplus H_0(V \setminus D_{\mathbb{R}}) \to H_0(V) \oplus H_0(V) ].
\end{equation}
Let
$$ H_0(D_{\mathbb{R}}) \oplus H_0(D_{\mathbb{R}}) \ni \theta + \lambda= \sum_{j} \eta_j \{p_j \} + \sum_{\ell} \mu_{\ell} \{ \psi_\ell \},$$
where $ p_j, \psi_\ell \in D_\mathbb{R}$. For a sufficiently small $ \varepsilon$ we have
\begin{eqnarray*}
H_0(D_{\mathbb{R}}) \oplus H_0(D_{\mathbb{R}})& \ni & \theta+ \lambda \mapsto \biggl( \sum_{j} \eta_j( \{p_j- \varepsilon \}- \{p_j+ \varepsilon \})+ \sum_{\ell} \mu_{\ell}( \{ \psi_\ell- \varepsilon \}- \{ \psi_\ell+ \varepsilon \}) \biggl)\\
& \in & \hbox{ker}[H_0(V \setminus D_\mathbb{R}) \oplus H_0(V \setminus D_{\mathbb{R}}) \to H_0(V) \oplus H_0(V) ].
\end{eqnarray*}
Thus, we prove \eqref{hom8}.
\\Let $ \eta + \mu=\sum_{j} \eta_j( \{p_j- \varepsilon \}- \{p_j+ \varepsilon \})+ \sum_{\ell} \mu_{\ell}( \{ \psi_\ell- \varepsilon \}- \{ \psi_\ell+ \varepsilon \}) \in \hbox{ker}[H_0(V \setminus D_\mathbb{R}) \oplus H_0(V \setminus D_{\mathbb{R}}) \to H_0(V) \oplus H_0(V) ]$. We can describe the homomorphism to $H_0(D \setminus D_ {\mathbb{R}}) \oplus H_0(D \setminus D_{\mathbb{R}})$ as
$$ \eta + \mu \mapsto \biggl( \sum_j \eta_j+ \sum_\ell \mu_\ell, -\sum_j \eta_j - \sum_\ell \mu_\ell \biggl) \in \mathbb{Z}^2 \oplus \mathbb{Z}^2 \simeq H_{0}(D \setminus D_{\mathbb{R}}) \oplus H_{0}(D \setminus D_{\mathbb{R}}).$$
Therefore
$$ F \simeq \widetilde{H}_0(D_\mathbb{R}) \oplus \widetilde{H}_0(D_\mathbb{R}).$$
For any domain $M \subset \mathbb{C}$, by the K\"unneth formula and Proposition \ref{split3} we have
\begin{equation}
\label{hom7p}
H_2(M \times \mathbb{S}_{\mathbb{R}_3})=H_0(M) \oplus H_0(M).
\end{equation}
By the same homotopy equivalences used in Proposition \ref{res21} (see \eqref{n6}), formula \eqref{hom7p} and the fact $H_{k}(D_{\mathbb{R}})={0}$ for $k>0$ we have
\begin{eqnarray*}
F' &\simeq& \hbox{ker}[H_2(\Omega^4_{W^*}) \to H_2(\Omega^4_{D^*}) \oplus H_2(\Omega^4_W)] \\
&\simeq & \hbox{ker}[H_0(D_\mathbb{R}) \oplus H_0(D_\mathbb{R}) \to H_0(D^+) \oplus H_0(D^+)] \\
&\simeq& \widetilde{H}_0(D_\mathbb{R}) \oplus \widetilde{H}_0(D_\mathbb{R}).
\end{eqnarray*}
Therefore
$$ F' \simeq \widetilde{H}_0(D_\mathbb{R}) \oplus \widetilde{H}_0(D_\mathbb{R}).$$
Moreover by \eqref{n10} we get
\begin{eqnarray*}
H_{3}((V\setminus D_{\mathbb{R}}) \times \mathbb{S}_{\mathbb{R}_3}) &\simeq& H_1(V \setminus D_{\mathbb{R}}) \oplus H_1(V \setminus D_{\mathbb{R}}) \simeq H_{1}\left(D^+ \cap (V \setminus D_{\mathbb{R}}) \right) \oplus H_{1}\left(D^- \cap (V \setminus D_{\mathbb{R}}) \right) \oplus \\
&& \oplus H_{1}\left(D^+ \cap (V \setminus D_{\mathbb{R}}) \right) \oplus H_{1}\left(D^- \cap (V \setminus D_{\mathbb{R}}) \right) \simeq H_{1}(D_{\mathbb{R}}) \oplus H_{1}(D_{\mathbb{R}}) \oplus\\
&& \oplus H_{1}(D_{\mathbb{R}}) \oplus H_{1}(D_{\mathbb{R}})  \simeq 0.
\end{eqnarray*}

Putting together all these facts we obtain

$$
\footnotesize
\begin{tikzcd}
0 \arrow[r, ""] \arrow[d, ""']
& H_1(D^+) \oplus H_1(D^-) \oplus H_1(D^+) \oplus H_1(D^-) \arrow[r, "\eta_1"] \arrow[d, "\rho_1"']
& H_1(D) \oplus H_1(D) \arrow[r, "\eta_2"] \arrow[d, "\rho_2"]
& \widetilde{H}_0(D_\mathbb{R}) \oplus \widetilde{H}_0(D_\mathbb{R}) \arrow[r, ""] \arrow[d, "\rho_3=Id"]
& 0 \ \arrow[d, ""]\\
0 \arrow[r, ""] & H_1(D^+) \oplus H_1(D^+) \arrow[r, "\mu_1"] & H_3(\Omega^4_D) \arrow[r, "\mu_2"] &  \widetilde{H}_0(D_\mathbb{R}) \oplus \widetilde{H}_0(D_\mathbb{R}) \arrow[r, ""] & 0
\end{tikzcd}
$$

The homomorphism $ \rho_1$ is induced by the embedding
$$(D \setminus D_\mathbb{R}) \oplus (D \setminus D_\mathbb{R})=(D^+ \cup D^- ) \oplus (D^+ \cup D^-) \to \Omega^4_{D^*},$$
and
$$ H_3(\Omega^4_{D^*}) \simeq H_3(D^+ \times \mathbb{S}_{\mathbb{R}_3}) \simeq H_1(D^+) \oplus H_1(D^+).$$
Thus we can define
$$ \rho_1(c_1, c_2, c_3, c_4)=(c_1+\tau_*c_2, c_3+ \tau_*c_4)$$
where $c_1$ and $c_3$ are 1-cycles in $D^+$ and $c_2, c_4$ are 1-cycles in $D^-$. In particular $\rho_1$ is surjective and $ \hbox{ker} (\rho_1)= \{(c, -\tau^*c,c', -\tau_*c'): \, \, c, c' \in H_{1}(D^+) \}$. Moreover, for any domain $M \subset \mathbb{C}$ we have
\begin{equation}
H_3(M \times \mathbb{S}^2) \simeq H_1(M).
\end{equation}
This implies that $ \rho_2$ is defined by
$$ \rho_2: H_1(D) \oplus H_1(D) \simeq H_3(D \times \mathbb{S}^2) \oplus H_3(D \times \mathbb{S}^2) \overset{\zeta_*}{\to} H_3(\Omega^4_D).$$
We set $ \beta'= \rho_2$ and define $ \alpha'(c,c')= \eta_1(c, -\tau^*c,c', -\tau_*c')$. Due to the exactness of the sequence, $ \eta_1$ is injective, so $ \alpha'$ is injective too.
\\ Now, we prove that $ \beta'$ is surjective.
\\ Let $s \in H_3(\Omega^4_D)$, since $ \rho_3$ is an isomorphism we can find a couple $(b,b_1) \in H_1(D) \oplus H_1(D)$ with $ \eta_2(b, b_1)= \mu_2(s)$. Then, by the commutative diagram
$$ \mu_2 \bigl(s- \rho_2(b,b_1) \bigl)= \mu_2(s)- \mu_2( \rho_2(b,b_1))= \mu_2(s)- \eta_2(b,b_1)=0,$$
then by the exactness of the sequence we get $ s- \rho_2(b,b_1) \in \hbox{ker}(\mu_2)= \hbox{Im} (\mu_1)$.
Since $ \rho_1$ is surjective there exists $(a, a_1) \in H_1(D^+) \oplus H_1(D^-) \oplus H_1(D^+) \oplus H_1(D^-)$ such that
$$ s- \rho_2(b,b_1)= \mu_1 \bigl(\rho_1(a, a_1) \bigl)= \rho_2\bigl(\eta_1(a, a_1) \bigl).$$
Thus, $s= \rho_2[(b,b_1)+ \eta_1(a,a_1)]$. This means that $ \beta'$ is surjective.
\\ In order to prove the exactness of the sequence \eqref{wseq15} we have to prove that $ \hbox{Im} (\alpha')= \hbox{ker} (\beta')$. We show the equality by double inclusion. Firstly we demonstrate $ \hbox{Im} (\alpha') \subseteq \hbox{ker} (\beta')$. By the commutative diagram, the previous definitions of $\alpha'$ and $ \rho_1$ and the fact that $ \tau_* \tau_*=Id$ we obtain
\begin{eqnarray*}
\beta'(\alpha'(b,b_1)) &=& \rho_2(\alpha'(b,b_1))= \rho_2(\eta_1(b,-\tau_* b, b_1,-\tau_* b_1)) \\
&=& \mu_1(\rho_1(b,-\tau_* b, b_1,-\tau_* b_1))\\
&=& \mu_1(b-\tau_* \tau_* b, b_1-\tau_* \tau_* b_1)\\
&=& \mu_1(0,0)=0.
\end{eqnarray*}
Now, let us prove that $ \hbox{ker}( \beta') \subseteq \hbox{Im} (\alpha')$. Let us assume that the couple $(b,b_1) \in \hbox{ker} (\beta')$, so $ \beta'(b,b_1)=\rho_2(b,b_1)=0$. By the commutative diagram and the fact that $\rho_3=Id$ we have
$$ \mu_2(\rho_2(b,b_1))= \rho_3(\eta_2(b,b_1))= \eta_2(b,b_1).$$
On the other side
$ \mu_2(\rho_2(b,b_1))= \mu_2(0)=0$.
Thus $ \eta_2(b,b_1)=0$. This means by the exactness of the sequence that $(b,b_1) \in \hbox{ker} (\eta_2)= \hbox{Im} (\eta_1)$. Then there exists $(c', c'', c''', c^{iv}) \in H_1(D^+) \oplus H_1(D^-) \oplus H_1(D^+) \oplus H_1(D^-)$ such that 
\begin{equation}
\label{n7}
\eta_1(c', c'', c''', c^{iv})=(b,b_1).
\end{equation}
Due to the exactness of the sequence, $\mu_1$ is injective; this implies that $\rho_1(c', c'', c''', c^{iv})=0$ because
$$ \mu_1(\rho_1(c', c'', c''', c^{iv}))= \rho_2(\eta_1(c', c'', c''', c^{iv}))= \rho_2(b,b_1)=0.$$
Hence
$$ 0= \rho_1(c', c'', c''', c^{iv})=(c'+\tau_*(c''), c'''+ \tau_*(c^{iv})),$$
so we have $c'=- \tau_* c''$ and $c'''=- \tau_{*}c^{iv}$ . Therefore by \eqref{n7} we have
$$ (b,b_1)= \eta_1(c', -\tau_{*} c',c''', -\tau_{*} c''')= \alpha'(c',c''').$$
This implies $ (b,b_1) \in \hbox{Im} (\alpha')$.
\end{proof}
\begin{Cor}
\label{res4}
Let $D \subset D_1$ be symmetric open subsets in $ \mathbb{C}$ such that $H_1(\Omega_D^4) \to H_1(\Omega_{D_1}^4)$ and $H_5(\Omega_D^4) \to H_5(\Omega_{D_1}^4)$ are injective simultaneously. Then $H_1(D) \to H_1(D_1)$ is injective.
\end{Cor}
\begin{proof}
From Proposition \ref{desc} we know that
$$ H_1(\Omega_D^4) \simeq H_1(D^+),$$
$$ H_1(\Omega_{D_1}^4) \simeq H_1(D^+_1).$$
By the inclusion $D \hookrightarrow D_1$ and Proposition \ref{res21} we get the following commutative diagram
$$
\begin{tikzcd}
0 \arrow[r, ""] \arrow[d, ""]
& H_1(\Omega_D^4) \arrow[r, ""] \arrow[d, "f_1"']
& H_1(D) \arrow[r, ""] \arrow[d, ""]
& H_5(\Omega_D^4) \arrow[r, ""] \arrow[d, "f_2"]
& 0 \ \arrow[d, ""]\\
0 \arrow[r, ""] & H_1(\Omega_{D_1}^4) \arrow[r, ""] & H_1(D_1) \arrow[r, ""] & H_5(\Omega_{D_1}^4) \arrow[r, ""] & 0
\end{tikzcd}
$$
Since $f_1$ and $f_2$ are injective by hypothesis, the snake lemma yields the thesis.
\end{proof}

\begin{Cor}
\label{res41}
Let $D \subset D_1$ be symmetric open subsets in $ \mathbb{C}$ such that $H_1(\Omega_D^4) \to H_1(\Omega_{D_1}^4)$ and $H_3(\Omega_D^4) \to H_3(\Omega_{D_1}^4)$ are injective simultaneously. Then $H_1(D) \to H_1(D_1)$ is injective.
\end{Cor}
\begin{proof}
By the inclusion $D \hookrightarrow D_1$ and Proposition \ref{res2b} we have the following commutative diagram
$$
\begin{tikzcd}
0 \arrow[r, ""] \arrow[d, ""']
& H_1(D^+) \oplus H_1(D^+) \arrow[r, ""] \arrow[d, "g_1"]
& H_1(D) \oplus H_1(D) \arrow[r, ""] \arrow[d, ""]
& H_3(\Omega_D^4) \arrow[r, ""] \arrow[d, "g_2"]
& 0 \arrow[d, ""]\\
0 \arrow[r, ""] & H_1(D^+_1) \oplus H_1(D^+_1) \arrow[r, ""] & H_1(D_1) \oplus H_1(D_1) \arrow[r, ""] & H_3(\Omega_{D_1}^4) \arrow[r, ""] & 0
\end{tikzcd}
$$
By Propositon \ref{desc} we have that $ H_1(D^+) \simeq H_1(\Omega_D^4)$, $  H_1(D^+_1) \simeq H_1(\Omega_{D_1}^4)$. Hence $H_1(D^+) \oplus H_1(D^+) \to H_1(D^+_1) \oplus H_1(D^+_1)$ is injective beacuse by hypothesis $H_1(\Omega_D^4) \to H_1(\Omega_{D_1}^4)$ is injective. Finally, we get the thesis by the snake lemma, since $g_1$ and $g_2$ are injective.
\end{proof}

\begin{lemma}
\label{res5}
Let $K$ be a symmetric compact connected subset of $ \mathbb{C}$ such that $K \cap \mathbb{R} \neq \emptyset$ and connected. Let $K'$ be a non-empty symmetric closed subset of $K$ and define
$$ D:= \mathbb{C} \setminus K,$$
$$ D_1:= \mathbb{C} \setminus K'.$$
Then $H_5(\Omega_D^4) \to H_5(\Omega_{D_1}^4)$ is injective.
\end{lemma}
\begin{proof}
By construction $H_1(D) \simeq \mathbb{Z}$ and $ \widetilde{H}_0(D_{\mathbb{R}}) \simeq \mathbb{Z}$. By Lemma \ref{res2} we obtain $H_1(D^+) \simeq \{0\}$, so by the exact sequence \eqref{sec1} we get $H_5(\Omega_D^4) \simeq \mathbb{Z}$.  Let us consider $R_1 > \max \{ |x|: \,  x \in K \}$. We consider a 5-cycle $S$ with center 0 and radius $R_1$ in $ \mathcal{Q}_{\mathbb{R}_3}$. By this construction the set $K$ is in the interior of the 5-cycle $S$, this means that the 5-cycle defines a non-trivial homology class in $H_5(\Omega_D^4)$. It happens the same for $K'$. Therefore, we have that the map $i_* : H_{5}(\Omega_D^4) \to H_5(\Omega_{D_1}^4)$ maps a non- trivial element of $H_5(\Omega_D^4)$ to a non trivial element of $H_5(\Omega_{D_1}^4)$. This implies the thesis because $H_5(\Omega_D^4) \simeq \mathbb{Z}$.
\end{proof}

\begin{lemma}
\label{res5b}
Let $P$ and $Q$ be symmetric compact subset of $ \mathbb{C}$ such that $P \cap \mathbb{R} \neq \emptyset$, $Q \cap \mathbb{R} \neq \emptyset$ and connected. Moreover $Q \cap P = \emptyset$. Let $P'$ and $Q'$ be two non-empty symmetric closed subsets of $P$ and $Q$, respectively. Let us define
$$D:= (\mathbb{C} \setminus P)\setminus Q,$$
$$ D_1:= (\mathbb{C} \setminus P')\setminus Q'.$$
Then $H_3(\Omega_D^4) \to H_3(\Omega_{D_1}^4)$ is injective.
\end{lemma}
\begin{proof}
From the definition of the set $D$ we have
$$ H_{1}(D) \simeq \mathbb{Z} \oplus \mathbb{Z},$$
$$ H_{0}(D_{\mathbb{R}}) \simeq \mathbb{Z} \oplus \mathbb{Z} \oplus \mathbb{Z}.$$
From the last one we derive that
$$ \widetilde{H}_{0}(D_{\mathbb{R}}) \simeq \mathbb{Z} \oplus \mathbb{Z}.$$
By Lemma \ref{res2} we obtain $H_1(D^+) \simeq \{0\}$. Thus by the exact sequence \eqref{sec2} we get $H_3(\Omega_D^4) \simeq \mathbb{Z} \oplus \mathbb{Z} \oplus \mathbb{Z} \oplus \mathbb{Z}$.
\\ Now we define two different closed curves $\gamma_1$ and $ \gamma_2$ such that they do not intersect themselves and surround $P$ and $Q$, respectively. We remark that since $P'$ and $Q'$ are in the interior of $P$ and $Q$, respectively, they are also surrounded by $ \gamma_1$ and $\gamma_2$. This gives us the possibility to define the following inclusion map
$$ i_*:\left(\Omega^2_{{\gamma_1}}, \Omega^2_{{\gamma_2}}, \Omega^2_{{\gamma_1}}, \Omega^2_{{\gamma_2}} \right)_D \hookrightarrow \left(\Omega^2_{{\gamma_1}}, \Omega^2_{{\gamma_2}}, \Omega^2_{{\gamma_1}}, \Omega^2_{{\gamma_2}} \right)_{D_1},$$
where $ \Omega^2_{{\gamma_1}}$ and $\Omega^2_{{\gamma_2}}$ are $ \gamma_1 \times [\mathbb{S}^2]$ and $\gamma_2 \times [\mathbb{S}^2]$, respectively, and the subscripts outside the brackets recall the fact that we consider the closed curves $\gamma_1$ and $ \gamma_2$ in $D$ and $D_1$, respectively.
This means that
$$ i_*: H_{3}(\Omega_D^4) \to H_{3}(\Omega_{D_1}^4),$$
maps four non-trivial independent elements of $H_3(\Omega_D^4)$ to  four non-trivial independent elements of $H_3(\Omega_{D_1}^4)$. Finally, since $H_{3}(\Omega_D^4) \simeq \mathbb{Z} \oplus \mathbb{Z} \oplus \mathbb{Z} \oplus \mathbb{Z}$ we get that $H_{3}(\Omega_D^4) \to H_3(\Omega_{D_1}^4)$ is injective.
\end{proof}

\begin{nb}
In the previous lemma the hypothesis $P \cap Q= \emptyset$ is essential to have two different holes in the complex plane.
\end{nb}

\begin{nb}
It is possible to write the sets $D$ and $D_1$ of the previous lemma in other ways, such as $D:= (\mathbb{C} \setminus P) \cap (\mathbb{C} \setminus Q)$ and $ D_1:=( \mathbb{C} \setminus P') \cap (\mathbb{C} \setminus Q').$
\end{nb}
\begin{prop}
\label{res8}
Let $D \subset D_1$ be symmetric open subset of $ \mathbb{C}$ such that $H_1(D) \to H_1(D_1)$ is injective. Then $H_5(\Omega_D^4) \to H_{5}(\Omega_{D_1}^4)$ is injective.
\end{prop}
\begin{proof}
We show this result by absurd. Let
$$ \delta \in \hbox{ker}(H_{5}(\Omega_D^4) \to H_{5}(\Omega_{D_1}^4)), \qquad \delta \neq 0.$$
The injectivity of $H_1(D) \to H_1(D_1)$ implies by Corollary \ref{res3} that $H_1(D^+) \to H_1(D_1^+)$ is injective, too. From the inclusion map $D \hookrightarrow D_1$ applied to the sequence \eqref{sec1b} we get the following commutative diagram
$$
\begin{tikzcd}
0 \arrow[r, ""] \arrow[d, ""']
& H_1(D^+) \arrow[r, ""] \arrow[d, ""']
& H_5(\Omega_D^4) \arrow[r, ""] \arrow[d, ""]
& \widehat{H}_0(D_{\mathbb{R}}) \arrow[r, ""] \arrow[d, ""]
& 0 \ \arrow[d, ""]\\
0 \arrow[r, ""] & H_1(D^+_1) \arrow[r, ""] & H_5(\Omega_{D_1}^4) \arrow[r, ""] &  \widehat{H}_0(D_{1, \, \mathbb{R}}) \arrow[r, ""] & 0
\end{tikzcd}
$$
Following similar computations of \cite[Prop. 2.14]{BW1} we reach the absurd.
\end{proof}
\begin{prop}
\label{res8b}
Let $D \subset D_1$ be symmetric open subsets of $ \mathbb{C}$ such that $H_1(D) \oplus H_1(D) \to H_1(D_1) \oplus H_1(D_1)$ is injective. Then $H_3(\Omega_D^4) \to H_3(\Omega_{D_1}^4)$ is injective.
\end{prop}

\begin{proof}
By absurd let us assume
$$ \alpha \in \hbox{ker}(H_3(\Omega_D^4) \to H_{3}(\Omega_{D_1}^4)), \quad \alpha \neq 0.$$	
By the hypothesis we derive that $H_1(D) \to H_1(D_1)$ is injective, thus by Corollary \ref{res3} we have that $H_1(D^+) \to H_1(D_1^+)$ is injective, it follows that $H_1(D^+) \oplus H_1(D^+) \to H_1(D_1^+) \oplus H_1(D_1^+) $ is injective, too.
\\ The inclusion map $D \hookrightarrow D_1$ applied to the sequence \eqref{sec2b} yields the following commutative diagram
$$
\begin{tikzcd}
0 \arrow[r, ""] \arrow[d, ""']
& H_1(D^+) \oplus H_1(D^+) \arrow[r, ""] \arrow[d, ""']
& H_3(\Omega_D^4) \arrow[r, ""] \arrow[d, ""]
& \widehat{H}_0(D_{\mathbb{R}}) \oplus \widehat{H}_0(D_{\mathbb{R}}) \arrow[r, ""] \arrow[d, ""]
& 0 \ \arrow[d, ""]\\
0 \arrow[r, ""] & H_1(D^+_1) \oplus H_1(D^+_1) \arrow[r, ""] & H_3(\Omega_{D_1}^4) \arrow[r, ""] &  \widehat{H}_0(D_{1, \, \mathbb{R}}) \oplus \widehat{H}_0(D_{1, \, \mathbb{R}}) \arrow[r, ""] & 0
\end{tikzcd}
$$	
	
Let the couple $ ( \alpha_0, \alpha_0)$ be the image of $ \alpha$ in $\widehat{H}_0(D_{\mathbb{R}}) \oplus \widehat{H}_0(D_{\mathbb{R}})$. Now, we want to prove that $( \alpha_0, \alpha_0) \neq (0,0)$. In order to reach an absurd we suppose that $ (\alpha_0, \alpha_0)=(0,0)$. This means that $ \alpha$ is induced by a couple $(\beta, \beta) \in H_1(D^+) \oplus H_1(D^+)$. Thus, if $ \alpha \neq 0$ then $(\beta, \beta) \neq 0$. However, we have an absurd since $H_1(D^+) \oplus H_1(D^+) \to H_1(D^+_1) \oplus H_1(D^+_1)$ and $H_1(D_1^+) \oplus H_1(D_1^+) \to H_3(\Omega_D^4)$ are injective and $ \alpha$ is mapped to zero in $H_3(\Omega_{D_1}^4)$. Hence $( \alpha_0, \alpha_0) \neq (0,0).$
\\ Now, since the image of $ \alpha$ in $H_3(\Omega_{D_1}^4)$ is zero we have that the image in $ \widehat{H}_0(D_{1, \, \mathbb{R}}) \oplus \widehat{H}_0(D_{1, \, \mathbb{R}})$ is zero, too. Thus implies that $ \alpha_0$ vanishes in $\widehat{H}_0(D_{1, \mathbb{R}})$. In particular we obtain that
$$ \alpha_0 \in \hbox{ker}(\widehat{H}_0(D_{\mathbb{R}}) \to \widehat{H}_0(D_{1, \mathbb{R}})).$$
We can represent $ \alpha_0$ as a formal $ \mathbb{Z}$-linear combination $ \sum_{x \in I} n_{x} \{ x \}$, where $I$ is a finite subset of $D_{\mathbb{R}}$. Moreover, by the definition of $\widehat{H}_{0}(D_{\mathbb{R}})$ (see Corollary \ref{result1}) we have that
$$ \alpha_0 \in \hbox{ker}(H_0(D_{\mathbb{R}}) \to H_0(D^+)).$$
This implies that $ \sum_{k} n_k=0$. Since $( \alpha_0, \alpha_0) \neq (0,0)$ and the previous facts we can find two points $q \in \mathbb{R} \setminus D$ and $q' \in \mathbb{R} \setminus D$ ($ q \neq q'$) such that
$$ \sum_{p \in I, \, p>q} n_p \neq 0, \qquad \sum_{p' \in I, \, p'>q'} n_{p'} \neq 0.$$
In order to fix the ideas, let $q< q'.$
\begin{figure}[H]
\centering
\resizebox{0.90\textwidth}{!}{%
\tikzset{every picture/.style={line width=0.75pt}} %set default line width to 0.75pt        

\begin{tikzpicture}[x=0.75pt,y=0.75pt,yscale=-1,xscale=1]
%uncomment if require: \path (0,386); %set diagram left start at 0, and has height of 386

%Shape: Polygon Curved [id:ds40504620091452437] 
\draw  [fill={rgb, 255:red, 18; green, 130; blue, 208 }  ,fill opacity=0.97 ] (105.33,76.75) .. controls (114.83,67.58) and (209.42,32.08) .. (274.33,27.25) .. controls (339.25,22.42) and (436.33,15.25) .. (486.33,53.75) .. controls (536.33,92.25) and (536.5,153.85) .. (544.33,206.75) .. controls (552.16,259.65) and (494.93,319.43) .. (457.33,342.75) .. controls (419.74,366.07) and (310.28,355.16) .. (251.33,357.75) .. controls (192.39,360.34) and (182.07,340.73) .. (155.42,332.58) .. controls (128.77,324.44) and (87.33,259.42) .. (81.33,245.75) .. controls (75.33,232.08) and (64.71,163.58) .. (72.33,136.75) .. controls (79.96,109.92) and (95.83,85.92) .. (105.33,76.75) -- cycle ;
%Shape: Polygon Curved [id:ds811068871137095] 
\draw  [fill={rgb, 255:red, 139; green, 181; blue, 93 }  ,fill opacity=1 ] (175.33,88.75) .. controls (201.33,75.75) and (236.59,81.95) .. (276.42,85.58) .. controls (316.25,89.22) and (456.73,66.31) .. (478.33,82.75) .. controls (499.94,99.19) and (515.42,222.08) .. (507.33,258.75) .. controls (499.25,295.42) and (457.33,316.25) .. (429.33,321.25) .. controls (401.33,326.25) and (307.57,324.51) .. (257.42,326.58) .. controls (207.27,328.66) and (173.81,302.52) .. (157.75,289) .. controls (141.69,275.48) and (120.42,249.08) .. (117.42,222.08) .. controls (114.42,195.08) and (111.33,158.75) .. (118.33,139.75) .. controls (125.33,120.75) and (149.33,101.75) .. (175.33,88.75) -- cycle ;
%Shape: Ellipse [id:dp7755003275497575] 
\draw  [fill={rgb, 255:red, 227; green, 211; blue, 18 }  ,fill opacity=1 ] (234.25,279) .. controls (212.43,279) and (194.75,246.32) .. (194.75,206) .. controls (194.75,165.68) and (212.43,133) .. (234.25,133) .. controls (256.07,133) and (273.75,165.68) .. (273.75,206) .. controls (273.75,246.32) and (256.07,279) .. (234.25,279) -- cycle ;
%Shape: Circle [id:dp7801097900707478] 
\draw  [fill={rgb, 255:red, 255; green, 255; blue, 255 }  ,fill opacity=1 ] (230.42,161.75) .. controls (230.42,155.17) and (235.75,149.83) .. (242.33,149.83) .. controls (248.91,149.83) and (254.25,155.17) .. (254.25,161.75) .. controls (254.25,168.33) and (248.91,173.67) .. (242.33,173.67) .. controls (235.75,173.67) and (230.42,168.33) .. (230.42,161.75) -- cycle ;
%Shape: Circle [id:dp41026949008960656] 
\draw  [fill={rgb, 255:red, 255; green, 255; blue, 255 }  ,fill opacity=1 ] (227,248.92) .. controls (227,242.34) and (232.34,237) .. (238.92,237) .. controls (245.5,237) and (250.83,242.34) .. (250.83,248.92) .. controls (250.83,255.5) and (245.5,260.83) .. (238.92,260.83) .. controls (232.34,260.83) and (227,255.5) .. (227,248.92) -- cycle ;
%Curve Lines [id:da6403824419034345] 
\draw    (176.43,213.64) .. controls (177.8,204.94) and (179.4,196.8) .. (181.2,189.19) .. controls (208.38,74.33) and (281.68,83.47) .. (300.75,209.46) ;
%Curve Lines [id:da734782588860902] 
\draw    (175.39,211.67) .. controls (175.99,214.1) and (176.61,216.46) .. (177.25,218.77) .. controls (211.49,342.47) and (293.88,292.16) .. (300.75,211.67) ;
%Shape: Circle [id:dp21568174346981084] 
\draw  [fill={rgb, 255:red, 0; green, 0; blue, 0 }  ,fill opacity=1 ] (173.46,213.64) .. controls (173.44,212.55) and (174.3,211.67) .. (175.39,211.67) .. controls (176.48,211.67) and (177.38,212.55) .. (177.41,213.64) .. controls (177.43,214.73) and (176.57,215.61) .. (175.48,215.61) .. controls (174.39,215.61) and (173.49,214.73) .. (173.46,213.64) -- cycle ;
%Shape: Circle [id:dp37003958801106673] 
\draw  [fill={rgb, 255:red, 0; green, 0; blue, 0 }  ,fill opacity=1 ] (298.54,211.67) .. controls (298.54,210.45) and (299.53,209.46) .. (300.75,209.46) .. controls (301.97,209.46) and (302.96,210.45) .. (302.96,211.67) .. controls (302.96,212.89) and (301.97,213.88) .. (300.75,213.88) .. controls (299.53,213.88) and (298.54,212.89) .. (298.54,211.67) -- cycle ;
%Shape: Circle [id:dp7114402332685407] 
\draw  [fill={rgb, 255:red, 0; green, 0; blue, 0 }  ,fill opacity=1 ] (237.04,212.04) .. controls (237.04,210.78) and (238.07,209.75) .. (239.33,209.75) .. controls (240.6,209.75) and (241.62,210.78) .. (241.62,212.04) .. controls (241.62,213.31) and (240.6,214.33) .. (239.33,214.33) .. controls (238.07,214.33) and (237.04,213.31) .. (237.04,212.04) -- cycle ;
%Curve Lines [id:da9971897536901311] 
\draw    (238.33,255.75) .. controls (255.46,227.17) and (245.83,222.17) .. (239.33,210.75) ;
%Curve Lines [id:da08060134015233944] 
\draw    (239.33,210.75) .. controls (256.46,182.17) and (248.83,173.17) .. (242.33,161.75) ;
%Curve Lines [id:da29422824075067366] 
\draw    (340.43,214.64) .. controls (341.8,205.94) and (343.4,197.8) .. (345.2,190.19) .. controls (372.38,75.33) and (445.68,84.47) .. (464.75,210.46) ;
%Curve Lines [id:da8603691505247034] 
\draw    (340.39,210.46) .. controls (340.99,212.89) and (341.61,215.26) .. (342.25,217.56) .. controls (342.89,219.87) and (397.33,396.75) .. (464.75,210.46) ;
%Shape: Ellipse [id:dp9320127364823451] 
\draw  [fill={rgb, 255:red, 227; green, 211; blue, 18 }  ,fill opacity=1 ] (401.25,276) .. controls (379.43,276) and (361.75,243.32) .. (361.75,203) .. controls (361.75,162.68) and (379.43,130) .. (401.25,130) .. controls (423.07,130) and (440.75,162.68) .. (440.75,203) .. controls (440.75,243.32) and (423.07,276) .. (401.25,276) -- cycle ;
%Shape: Circle [id:dp6459088044250167] 
\draw  [fill={rgb, 255:red, 255; green, 255; blue, 255 }  ,fill opacity=1 ] (388.42,156.75) .. controls (388.42,150.17) and (393.75,144.83) .. (400.33,144.83) .. controls (406.91,144.83) and (412.25,150.17) .. (412.25,156.75) .. controls (412.25,163.33) and (406.91,168.67) .. (400.33,168.67) .. controls (393.75,168.67) and (388.42,163.33) .. (388.42,156.75) -- cycle ;
%Shape: Circle [id:dp2923894829154724] 
\draw  [fill={rgb, 255:red, 255; green, 255; blue, 255 }  ,fill opacity=1 ] (390.42,251.75) .. controls (390.42,245.17) and (395.75,239.83) .. (402.33,239.83) .. controls (408.91,239.83) and (414.25,245.17) .. (414.25,251.75) .. controls (414.25,258.33) and (408.91,263.67) .. (402.33,263.67) .. controls (395.75,263.67) and (390.42,258.33) .. (390.42,251.75) -- cycle ;
%Curve Lines [id:da7309080144929058] 
\draw    (401.33,211.75) .. controls (418.46,183.17) and (403.75,164.42) .. (397.25,153) ;
%Curve Lines [id:da7443808449730009] 
\draw    (400.33,256.75) .. controls (417.46,228.17) and (407.83,223.17) .. (401.33,211.75) ;
%Straight Lines [id:da5200915355854898] 
\draw    (22.33,212.64) -- (176.43,212.64) -- (256.35,212.64) -- (576.33,212.64) ;
%Shape: Circle [id:dp7942382022579101] 
\draw  [fill={rgb, 255:red, 0; green, 0; blue, 0 }  ,fill opacity=1 ] (467.33,212.58) .. controls (467.33,210.93) and (465.99,209.58) .. (464.33,209.58) .. controls (462.68,209.58) and (461.33,210.93) .. (461.33,212.58) .. controls (461.33,214.24) and (462.68,215.58) .. (464.33,215.58) .. controls (465.99,215.58) and (467.33,214.24) .. (467.33,212.58) -- cycle ;
%Shape: Circle [id:dp5102515393872786] 
\draw  [fill={rgb, 255:red, 0; green, 0; blue, 0 }  ,fill opacity=1 ] (343.43,213.64) .. controls (343.43,211.98) and (342.09,210.64) .. (340.43,210.64) .. controls (338.78,210.64) and (337.43,211.98) .. (337.43,213.64) .. controls (337.43,215.29) and (338.78,216.64) .. (340.43,216.64) .. controls (342.09,216.64) and (343.43,215.29) .. (343.43,213.64) -- cycle ;
%Shape: Circle [id:dp49463834105048043] 
\draw  [fill={rgb, 255:red, 0; green, 0; blue, 0 }  ,fill opacity=1 ] (405.33,212.75) .. controls (405.33,211.09) and (403.99,209.75) .. (402.33,209.75) .. controls (400.68,209.75) and (399.33,211.09) .. (399.33,212.75) .. controls (399.33,214.41) and (400.68,215.75) .. (402.33,215.75) .. controls (403.99,215.75) and (405.33,214.41) .. (405.33,212.75) -- cycle ;

% Text Node
\draw (34,194.83) node [anchor=north west][inner sep=0.75pt]   [align=left] {$\mathbb{R}$};
% Text Node
\draw (276.33,30.25) node [anchor=north west][inner sep=0.75pt]   [align=left] {{\small $D_1$}};
% Text Node
\draw (141,121.83) node [anchor=north west][inner sep=0.75pt]   [align=left] {D};
% Text Node
\draw (202,191.83) node [anchor=north west][inner sep=0.75pt]  [font=\small] [align=left] {B};
% Text Node
\draw (270,102.5) node [anchor=north west][inner sep=0.75pt]   [align=left] {$\gamma$};
% Text Node
\draw (159,215.5) node [anchor=north west][inner sep=0.75pt]  [font=\small] [align=left] {{\small $R_1$}};
% Text Node
\draw (280.96,214.67) node [anchor=north west][inner sep=0.75pt]  [font=\small] [align=left] {{\footnotesize $R_2$}};
% Text Node
\draw (250,193.67) node [anchor=north west][inner sep=0.75pt]  [font=\footnotesize] [align=left] {$\xi$};
% Text Node
\draw (226,213.67) node [anchor=north west][inner sep=0.75pt]  [font=\small] [align=left] {q};
% Text Node
\draw (320,217.67) node [anchor=north west][inner sep=0.75pt]   [align=left] {{\small $R'_1$}};
% Text Node
\draw (466.75,213.46) node [anchor=north west][inner sep=0.75pt]   [align=left] {{\small $R'_2$}};
% Text Node
\draw (388,217.67) node [anchor=north west][inner sep=0.75pt]   [align=left] {q'};
% Text Node
\draw (368,189.67) node [anchor=north west][inner sep=0.75pt]   [align=left] {B'};
% Text Node
\draw (413,183.67) node [anchor=north west][inner sep=0.75pt]   [align=left] {{\footnotesize $\xi'$}};
% Text Node
\draw (431,101.67) node [anchor=north west][inner sep=0.75pt]   [align=left] {$\gamma'$};

\end{tikzpicture}
}
\caption{Figure 1}
\end{figure}	
Fix such points $q$ and $q'$. Let us consider $B$ as the connected component of $ \mathbb{C} \setminus D$ containing $q$ and $B'$ the connected component  of $ \mathbb{C} \setminus D$ containing $q'$.
Fix $R_1$, $R_2 \in I$ and $R'_1$, $R'_2 \in I$ such that
$$ \begin{cases}
R_1 < q< R_2\\
R'_1< q' < R'_2
\end{cases}	
$$	
and $I \cap ]R_1,R_2[= \emptyset$, $I \cap ]R'_1,R'_2[ = \emptyset$. We know that $ \alpha_0$ is mapped to zero in $\widehat{H}_0(D_{1, \mathbb{R}})$, this means that $[R_1,R_2] \subset D_{1, \mathbb{R}}$ and $[R'_1,R'_2] \subset D_{1, \mathbb{R}}$. Moreover $ \alpha_0$ is mapped to zero in $H_{0}(D^+)$, this implies that both $R_1,R_2$ and $R'_1,R'_2$ are in the same connected component of $D^+$. Therefore $R_1$ and $R_2$ can be connected by a path $ \gamma$ in $D^+$. This path, combined with its image under conjugation yields, a closed curve inside $D$ which surrounds $q$. It is possible to repeat the same reasoning for $R_1'$ and $R'_2$. In this case we obtain a closed curve $ \gamma'$ inside $D$ which surrounds $q'$. Therefore $B$ and $B'$ must be bounded and $B \cap \mathbb{R} \subseteq ]R_1,R_2[$, $B' \cap \mathbb{R} \subseteq ]R'_1,R'_2[$. Furthermore, since $ [R_1, R_2] \subset D_{1, \mathbb{R}}$ and $[R'_1,R'_2] \subset D_{1, \mathbb{R}}$ by the previous facts $ \mathbb{R} \cap (B \setminus D_1)= \emptyset$ and $ \mathbb{R} \cap (B' \setminus D_1)= \emptyset$.
\\ By hypothesis $H_{1}(D) \oplus H_1(D) \to H_{1}(D_1) \oplus H_{1}(D_1)$ is injective. Thus, it maps two non-trivial independent 1-cycles of $H_1(D)$ to two non trivial 1-cycles of $H_1(D_1)$. This implies the boundedness of $B$ and $B'$ in $D_1$. Therefore, $B \cap D_1^c \neq \emptyset$ and $B' \cap D_1^c \neq \emptyset$.
\\Now we choose two paths
$$ \xi: [0,1] \to B$$
such that $\xi(0)=q$, $\xi(1) \notin D_1$, $ \xi(t) \notin \mathbb{R}$ and
$$ \xi': [0,1] \to B'$$
such that $\xi'(0)=q'$, $\xi'(1) \notin D_1$, $ \xi'(t) \notin \mathbb{R}$. Let us define
$$ P= \{z \in \mathbb{C}: \, \exists \,  t \in [0,1], \, z= \xi(t) \, \, \hbox{or} \, \, \overline{\xi(t)}\},$$
$$ Q= \{z \in \mathbb{C}: \, \exists \, t \in [0,1], \, z= \xi'(t) \, \, \hbox{or} \, \, \overline{\xi'(t)}\}.$$
We observe that $P \cap \mathbb{R}= \{ q\}$, $Q \cap \mathbb{R}= \{ q'\}$ and $P \cap Q= \emptyset$. Now, we consider the following diagram of inclusion maps
$$
\begin{tikzcd}
D \arrow[r, ""] \arrow[d, ""']
& D_1 \ \arrow[d, ""]\\
(\mathbb{C} \setminus P)\setminus Q \arrow[r, ""] & \bigl(\mathbb{C} \setminus (P \cap D_1^c)\bigl)  \setminus (Q \cap D_1^c)
\end{tikzcd}
$$
By Lemma \ref{res5b} we get the injectivity of the following map
\begin{equation}
\label{map}
H_{3}(\Omega_{[(\mathbb{C} \setminus P)\setminus Q]}^4 ) \to H_{3}(\Omega_{ \{[\mathbb{C} \setminus (P \cap D_1^c)]\setminus (Q \cap D_1^c)\} }^4).
\end{equation}
In this case
$$ D_{\mathbb{R}}:=[(\mathbb{C} \setminus P)\setminus Q] \cap \mathbb{R}= (\mathbb{R} \setminus P)\setminus Q.$$
By construction the couple $( \alpha_0, \alpha_0)$ is mapped to a non-zero element of $ \widehat{H}_{0}((\mathbb{R} \setminus P)\setminus Q) \oplus \widehat{H}_{0}((\mathbb{R} \setminus P)\setminus Q)$. By the sequence \eqref{sec2b} follows that $ \alpha$ is mapped to a non-zero element of $H_{3}(\Omega_{[(\mathbb{C} \setminus P)\setminus Q]}^4 )$. By assumption the image of $ \alpha$ is zero in $H_3(\Omega_{D_1}^4)$. Since $D_1 \subset   \bigl(\mathbb{C} \setminus (P \cap D_1^c)\bigl)  \setminus (Q \cap D_1^c)$, we have that its image in $\Omega^4_{ \{[\mathbb{C} \setminus (P \cap D_1^c)]\setminus (Q \cap D_1^c)\}}$ is zero, too. Finally, due to the injectivity of the map \eqref{map} we have an absurd.
\end{proof}
\begin{nb}
In the picture Figure 1 it is not possible to build a $1$-cycle, instead of two 1-cycles. The unique chance could be $R_2=R_1'$, but in general we are not sure if between $R_2$ and $R'_1$ there are any holes which intersect the real line.
\end{nb}
Finally, we can prove the implication $ 3) \Longleftrightarrow 4)$ of Theorem \ref{main}.
\begin{prop}
\label{res7}
Let $D \subset D_1$ be a symmetric open subset of $ \mathbb{C}$ with the corresponding axially symmetric subsets $ \Omega_D^4 \subset \Omega^4_{D_1}$ in $ \mathcal{Q}_{\mathbb{R}_3}$. Then $H_1(D) \to H_1(D_1)$ is injective if and only if $H_{1}(\Omega_D^4) \to H_1(\Omega_{D_1}^4)$, $H_{3}(\Omega_D^4) \to H_3(\Omega_{D_1}^4)$ and $H_{5}(\Omega_D^4) \to H_5(\Omega_{D_1}^4)$ are injective simultaneously.
\end{prop}
\begin{proof}
We can assume, without loss of generality, that $ \Omega_D$ is connected.
\\ If $H_{1}(\Omega_D^4) \to H_1(\Omega_{D_1}^4)$, $H_{3}(\Omega_D^4) \to H_3(\Omega_{D_1}^4)$ and $H_{5}(\Omega_D^4) \to H_5(\Omega_{D_1}^4)$ are injective, by Corollary \ref{res4} and Corollary \ref{res41} we have that $H_1(D) \to H_1(D_1)$ is injective.
\\ We assume that $H_1(D) \to H_1(D_1)$ is injective, in particular we have that $H_1(D) \oplus H_1(D) \to H_1(D_1) \oplus H_1(D_1)$ is injective. By Proposition \ref{res8} and Proposition \ref{res8b} we have that both $H_{3}(\Omega_D^4) \to H_3(\Omega_{D_1}^4)$ and $H_{5}(\Omega_D^4) \to H_5(\Omega_{D_1}^4)$ are injective. Due to Proposition \ref{desc} we have that $H_1(\Omega_D^4) \simeq H_1(D^+)$ and $H_1(\Omega_{D_1}^4) \simeq H_1(D^+_1)$, thus by Corollary \ref{res3}  we have that $H_{1}(\Omega_D^4) \to H_1(\Omega_{D_1}^4)$ is injective.
\end{proof}

\textbf{Acknowledgements}
\newline
We warmly thank Prof. Irene Sabadini for fruitful discussions.

\hspace{2mm}

\noindent
Cinzia Bisi,
Dipartimento di Matematica e Informatica \\Universit\`a di Ferrara\\
Via Ma\-chia\-vel\-li n.~30\\
I-44121 Ferrara\\
Italy

\noindent
\emph{email address}: bsicnz@unife.it\\
\emph{ORCID iD}: 0000-0002-4973-1053

\vspace*{5mm}
\noindent
Antonino De Martino,
Dipartimento di Matematica \\ Politecnico di Milano\\
Via Bonardi n.~9\\
20133 Milan\\
Italy

\noindent
\emph{email address}: antonino.demartino@polimi.it\\
\emph{ORCID iD}: 0000-0002-8939-4389

\vspace*{5mm}
\noindent
J\"org Winkelmann, Lehrstuhl Analysis II,
Fakult\"at f\"ur Mathematik \\ Ruhr-Universit\"at Bochum\\
44780 Bochum\\
Germany

\noindent
\emph{email address}: joerg.winkelmann@rub.de\\
\emph{ORCID iD}: 0000-0002-1781-5842

\end{document}